\documentclass{amsart}
\usepackage[normalem]{ulem}
\usepackage{amssymb,epsfig,mathrsfs,mathpazo,stackengine,scalerel,url,amsthm,thmtools,bbm,esint}
\usepackage{xcolor}
\usepackage{stmaryrd}
\usepackage{epsfig,bm}
\usepackage[multiple]{footmisc}
\usepackage{accents} % ensures that ${\bar{Y}^\delta}'$ doesn't give a double subscript error. 
\usepackage{mathtools} % for \mathclap
\usepackage{footmisc}
\usepackage{derivative}

\DeclareFontFamily{U}{mathx}{}
\DeclareFontShape{U}{mathx}{m}{n}{<-> mathx10}{}
\DeclareSymbolFont{mathx}{U}{mathx}{m}{n}
\DeclareMathAccent{\widehat}{0}{mathx}{"70}
\DeclareMathAccent{\widecheck}{0}{mathx}{"71}

\makeatletter
\newcommand*\bigcdot{\mathpalette\bigcdot@{.5}}
\newcommand*\bigcdot@[2]{\mathbin{\vcenter{\hbox{\scalebox{#2}{$\m@th#1\bullet$}}}}}
\makeatother

\usepackage{arydshln, booktabs, makecell, pbox, tabularx}

\usepackage{graphicx, caption, subcaption}

\newtheorem{theorem}{Theorem}[section]
\newtheorem{lemma}[theorem]{Lemma}
\newtheorem{proposition}[theorem]{Proposition}
\newtheorem{corollary}[theorem]{Corollary}

\newtheorem{algorithm}[theorem]{Algorithm}

\declaretheorem[style=definition,qed=$\vartriangle$,sibling=theorem]{example}
\declaretheorem[style=remark,qed=$\vartriangle$,sibling=theorem]{remark}
\numberwithin{equation}{section}

\newcommand{\NN}{N}

\newcommand{\R}{\mathbb R}

\newcommand{\N}{\mathbb N}

\newcommand{\cF}{\mathcal F}

\newcommand{\identity}{\mathrm{Id}}

\DeclareMathOperator*{\argmin}{argmin}

\DeclareMathOperator{\Span}{span}

\newcommand*\diff{\mathop{}\!\mathrm{d}}

\newcommand{\nrm}{| \! | \! |}

\newcommand{\be}{\begin{equation}}
\newcommand{\ee}{\end{equation}}

\newcommand{\uumlaut}{{\"u}}

\newenvironment{algotab}%
{\par\begin{samepage}%
\begin{tabbing}\ttfamily%
 \hspace*{5mm}\=\hspace{3ex}\=\hspace{3ex}\=\hspace{3ex}\=\hspace{3ex}%
\=\hspace{3ex}\=\hspace{3ex}\=\hspace{3ex}\=\hspace{3ex}\kill}%
{\end{tabbing}\end{samepage}}
\newcounter{ccondition}

{%
\renewcommand{\theequation}{\temp}%
\addtocounter{equation}{-1}%
\par\noindent%
\ignorespacesafterend%
}

\newcounter{mylistcounter}
\renewcommand{\themylistcounter}{(\roman{mylistcounter})}

\makeatletter
\@namedef{subjclassname@2020}{%
  \textup{2020} Mathematics Subject Classification}
\makeatother

\title[Quasi-optimal time-space discretizations for nonlinear parabolic PDEs]{Quasi-optimal time-space discretizations for a class of nonlinear parabolic PDEs}

\date{\today}

\author[N.\ Beranek]{Nina Beranek}
\author[R.K.H.\ Smeets]{Robin Smeets}
\author[R.P.\ Stevenson]{Rob Stevenson}
\address[N.\ Beranek]{Ulm University, Institute for Numerical Mathematics,
Helmholtzstr. 20, 89081 Ulm, Germany}
\email{nina.beranek@uni-ulm.de}
\address[R.K.H.\ Smeets and R.P.\ Stevenson]{Korteweg-de Vries (KdV) Institute for Mathematics, University of Amsterdam, PO Box 94248, 1090 GE Amsterdam, The Netherlands}
\email{r.k.h.smeets@uva.nl and r.p.stevenson@uva.nl}
\thanks{The work of Nina Beranek was supported by the Office for Gender Equality Ulm University.}

\subjclass[2020]{
35A15, %Variational methods applied to PDEs
35B35, % Stability in context of PDEs [See also 34Dxx, 37B25, 37C20, 37C75, 37F15, 37J25, 37K45, 37L15, 49K40, 58K25, 93Dxx]
35K90, % Abstract parabolic equations
65M12, % Stability and convergence of numerical methods for initial value and initial-boundary value problems involving PDEs
65M15, % Error bounds for initial value and initial- boundary value problems involving PDEs
65M60% Finite element, Rayleigh-Ritz and Galerkin methods for initial value and initial-boundary value problems involving PDEs
}

\keywords{Nonlinear parabolic PDEs, simultaneous time-space variational formulation, quasi-best approximation, uniform inf-sup stability, adaptive stabilization, inexact Uzawa iteration}

\begin{document}

\begin{abstract}
We consider parabolic evolution equations with Lipschitz continuous and strongly monotone spatial operators.
By introducing an additional variable, we construct an equivalent system where the operator is a Lipschitz continuous mapping from a Hilbert space $Y \times X$ to its dual, with a Lipschitz continuous inverse. Resulting Galerkin discretizations can be solved with an inexact Uzawa type algorithm.
Quasi-optimality of the Galerkin approximations is guaranteed under an inf-sup condition on the selected `test' and `trial' subspaces of $Y$ and $X$.
To circumvent the restriction imposed by this inf-sup condition, an a posteriori condition for quasi-optimality is developed that is shown to be satisfied whenever the test space is sufficiently large.
 \end{abstract}
\maketitle

%%%%%%%%%%%%%%
\section{Introduction}
%%%%%%%%%%%%%%
In recent years one witnesses a growing interest in time-space variational formulations of parabolic initial value problems (IVPs) as the basis for monolitic numerical solvers.
They enable to compute numerical approximations that are quasi-best without regularity assumptions on the solution \cite{11,249.992}, they lead to systems that have excellent scalability properties for parallel implementations \cite{169.06,306.65,169.065}, and they naturally allow adaptive local mesh refinements simultaneously in time and space.
At the downside, monolitic methods require more memory than time-stepping methods. This can be mitigated by using (adaptive) sparse tensor product approximation allowed by the Cartesian product structure of the time-space cylinder \cite{11,249.991}. The disadvantage of the larger memory requirement vanishes for applications where the forward parabolic problem is coupled with an adjoint backward parabolic problem, as with problems of optimal control \cite{169.053,19.96,75.292,75.2576}, or with goal-oriented error minimization \cite{70.48}.
Another interesting application of time-space methods is in reduced basis methods for parameter-dependent problems where, other than with time-stepping methods, they reduce 
complexity both in space \emph{and} time \cite{75.528,75.292}.

For \emph{linear} parabolic PDEs, where with a Gelfand triple $V \hookrightarrow H \simeq H' \hookrightarrow V'$,
for each time $t$ the linear elliptic spatial PDO (partial differential operator) is a mapping $V \rightarrow V'$, the commonly used time-space variational formulation writes the pair of the parabolic PDO and the initial trace operator 
as a boundedly invertible mapping $B_e \colon X \rightarrow Y' \times H$. Here $X:=L_2((0,T);V) \cap H^1((0,T);V')$ and $Y:=L_2((0,T);V)$. Well-posedness of the IVP can also be demonstrated for other choices of the spaces $X$ and $Y$, \cite{247.155,249.3}, or for its formulation as a first-order system \cite{75.257,75.28}, both with interesting numerical approximation schemes resulting from that.

For a \emph{nonlinear} parabolic PDE, where for each $t$ the spatial PDO is a Lipschitz continuous and strongly monotone mapping $V \rightarrow V'$ ,
\cite{318.3} shows that $B_e$ is an invertible mapping $X \rightarrow Y' \times H$. This result is shown under even relaxed assumptions.
To the best of our knowledge, however, any sort of stability of $B_e^{-1}$, essential for the development of numerical approximations, has not been demonstrated.
Under the conditions of strong monotonicity and Lipschitz continuity of the spatial PDO, in the current work it is shown that $B_e^{-1}$ is Lipschitz continuous.

In view of finding a formulation amenable to Galerkin discretizations, by introducing a secondary variable $\theta$ besides the solution $u$, we show that the IVP can be written as a $2 \times 2$ system of equations for $(\theta,u) \in Y \times X$, where both the operator $Y \times X \rightarrow Y' \times X'$
and its inverse are Lipschitz continuous. 
Our system formulation has been inspired by the application of the Br\'{e}zis-Ekeland-Nayroles variational principle \cite{35.81,233.5}, where we avoid the condition that the spatial PDO has a potential.

For finite dimensional subspaces $Y^\delta \subset Y$ and $X^\delta \subset X$, the latter equipped with a $Y^\delta$-dependent norm,
we show that the operator $Y^\delta \times X^\delta \rightarrow (Y^\delta)' \times (X^\delta)'$ resulting from a Galerkin discretization
 is Lipschitz continuous with a Lipschitz continuous inverse, with constants independent of $Y^\delta$ and $X^\delta$.
For pairs $X^\delta \subset Y^\delta$ that satisfy an inf-sup condition, uniformly over a family of such pairs, this shows that the second component of the Galerkin solution $(\theta^\delta,u^\delta) \in Y^\delta \times X^\delta$ is a quasi-best approximation from $X^\delta$ to the solution of the IVP w.r.t.~the norm on $X$.

This uniform inf-sup condition has been verified, with $\dim Y^\delta$ of the order of $\dim X^\delta$, for $Y^\delta$ and $X^\delta$ tensor products of function spaces of the temporal or spatial variables, or more generally, when this tensor product constraint holds true in each time-slab from an arbitrary partition of the time-interval.

The Galerkin approximation $(\theta^\delta,u^\delta)$ is the solution of a system of nonlinear equations that need to be solved.
We develop an Uzawa type algorithm, known for the iterative solution of linear saddle-point system, where 
the Lipschitz continuous and strongly monotone problems in the inner and outer loop are approximately solved by Zarantonello iterations, and demonstrate its R-linear convergence, uniformly in $Y^\delta$ and $X^\delta$.

Finally, although the setting of tensor product approximation per time-slab covers finite element spaces w.r.t. non-uniform, and different partitions in each time slab, it does not allow partitions that result from `local time stepping'.
In order to include fully general (finite element) spaces $X^\delta \subset Y^\delta$, and so to enable local (adaptive) mesh refinements,
inspired by an idea from \cite{45.44} we derive an a posteriori condition for quasi-optimality of the $u^\delta$ component of the Galerkin solution. This condition is shown to be satisfied whenever $Y^\delta$ is sufficiently large, which given an $X^\delta$, suggests an (adaptive) enrichment of $Y^\delta$ until the condition holds true. Validity of the condition moreover guarantees that an a posteriori computable estimator for $\|u-u^\delta\|_X$ is efficient and reliable, which suggests an outer (adaptive) enrichment of $X^\delta$.

A computational verification of the a posteriori condition requires an a posteriori error estimator for the error in a Galerkin approximation of the solution of an equation with an
Lipschitz continuous and strongly monotone operator $Y \rightarrow Y'$. 
Such an estimator and numerical experiments with the double adaptivity loop will be presented in forthcoming work.

\subsection{Organization} In Sect.~\ref{sec:2} we write the IVP in time-space variational form.
Some basic facts about Lipschitz continuous and strongly monotone operators are recalled in Sect.~\ref{sec:LipMono}.
In Sect.~\ref{sec:4} we write the IVP as an equivalent system, and show that the corresponding operator, as well as those corresponding to the Galerkin discretizations, are Lipschitz continuous with a Lipschitz continuous inverse. In Sect.~\ref{sec:5} it is shown that under an inf-sup condition the Galerkin approximations are quasi-optimal. Situations are recalled in which this inf-sup condition has been demonstrated. An alternative a posteriori condition for quasi-optimality is derived in Sect.~\ref{sec:Pjotr}. An inexact Uzawa algorithm for the iterative solution of the Galerkin systems is presented in Sect.~\ref{sec:7}. Conclusions and an outlook are given in Sect.~\ref{sec:8}.
For $Y^\delta$ and $X^\delta$ of tensor product form, in Appendix~\ref{sec:precond} a construction is given of optimal preconditioners that provide a cheap alternative for the application of the inverse Riesz operators in the Uzawa algorithm.

\subsection{Notations} For our convenience, we restrict ourselves to linear spaces over $\R$.
We equip any Cartesian product of normed spaces ${\mathcal X}$ and ${\mathcal Y}$ with $\|(x,y)\|:=\|x\|_{\mathcal X}+\|y\|_{\mathcal Y}$.
By $C \lesssim D$ we will mean that $C\ge0$ can be bounded by a multiple of $D\ge0$, independently of parameters on which $C$ and $ D$ may depend, in particular the discretization index $\delta$.
Obviously, $C \gtrsim D$ is defined as $D \lesssim C$, and $C\eqsim D$ as $C\lesssim D$ and $C \gtrsim D$.
For a Hilbert space $K$, $R_K\colon K \rightarrow K'$ will denote the Riesz isometry defined by $(R_K v)(w)=\langle w,v\rangle_K$.
%%%%%%%%%%%%%%

%%%%%%%%%%%%%%
\section{First-order evolution by Lipschitz continuous and strongly monotone operators} \label{sec:2}
%%%%%%%%%%%%%%

For a Gelfand triple $V \hookrightarrow H \simeq H' \hookrightarrow V'$, and a $T>0$, with $\Xi:=(0,T)$ we consider the \emph{Initial Value Problem} (IVP) of finding $u\colon \Xi \rightarrow V$ such that
\be \label{eq:IVP}
\left\{
\begin{array}{rcll}
\tfrac{\diff u}{\diff t} (t)+A(t)u(t) & = & \ell(t) & \text{a.e. }t \in \Xi  , \\
u(0) & = & u_0,
\end{array}
\right.
\ee
where $\ell(t) \in V'$, $A(t)\colon V \rightarrow V'$, and $u_0 \in H$.

The basic example for \emph{linear} $A(t)$ is given next. 
\begin{example}[Heat equation] \label{ex:simple} For a domain $\Omega \subset \R^d$, let $H:=L_2(\Omega)$, $V:=H^1_0(\Omega)$, and let 
$(A(t) \eta)(\xi):=\int_\Omega \nabla \eta(x) \bigcdot \nabla \xi(x) \diff x \quad (\eta,\xi \in H^1_0(\Omega))$.
\end{example}

Our main interest is in cases where $A(t)$ is \emph{nonlinear}.
We consider the IVP in \emph{operator}, or \emph{time-space variational form}: We set
$$
Y:=L_2(\Xi;V),
$$
$A\colon Y \rightarrow Y'$ by $(Aw)(v):=\int_\Xi (A(t)w(t))(v(t))\diff t$,
$$
X:=L_2(\Xi;V) \cap H^1(\Xi;V'),
$$
and $\diff_t\colon X \rightarrow Y'$ by $(\diff_t w)(v):=\int_\Xi (\tfrac{\diff w}{\diff t} (t))(v(t)) \diff t$.  It is known that $X \hookrightarrow C(\overline{\Xi},H)$, so that with the trace operator $\gamma_t \colon X \rightarrow H\colon w \mapsto w(t)$,
\be \label{eq:embedding}
C_{\Xi}:=\sup_{t \in \overline{\Xi}} \sup_{0 \neq w \in X} \frac{\|\gamma_t w\|_H}{\sqrt{\|w\|_Y^2+\|\diff_t w\|_{Y'}^2}}<\infty.
\ee

For smooth $u, w \in X$,
$$
(\diff_t u)(v)+(\diff_t v)(u)=\int_\Xi \tfrac{\diff}{\diff t} \langle u(t),v(t)\rangle_H \diff t=\langle u(T),v(T)\rangle_H-\langle u(0),v(0)\rangle_H,
$$
so that, by a density argument,
\be \label{eq:34}
\diff_t+\diff_t'+\gamma_0' \gamma_0=\gamma_T' \gamma_T
\ee
as mappings $X \rightarrow X'$.

Given data $(\ell,u_0) \in Y' \times H$, our IVP reads as finding $u \in X$ such that
$$
\left\{
\begin{array}{rcll}
(Bu)(v):=(\diff_t u)(v)+(Au)(v)&=&\ell(v) & (v \in Y),\\
\gamma_0 u &=& u_0,
\end{array}
\right.
$$
i.e., as
\be \label{eq:33}
B_e u:=\left[\begin{array}{@{}c@{}} B \\ \gamma_0 \end{array}\right] u=\left[\begin{array}{@{}c@{}} \ell \\ u_0\end{array}\right].
\ee

We will consider operators $A\colon Y \rightarrow Y'$ that are \emph{Lipschitz continuous} and \emph{strongly monotone}.
For such operators it is known that \eqref{eq:33} has a unique solution \cite[Theorem 30.A.]{318.3}.\footnote{This holds even true for $A$ that are only monotone, hemicontinuous, coercive, and bounded.} In the forthcoming Corollary~\ref{corol:Lipstab} it will be shown that this solution depends Lipschitz continuously on the data.

Our leading example is the following.

\begin{example}[Quasi-linear parabolic PDE] \label{ex:quasi}
As in Example~\ref{ex:simple}, let $H:=L_2(\Omega)$, $V:=H^1_0(\Omega)$, the latter equipped with $\|\nabla \bigcdot\|_{L_2(\Omega)^d}$, so that $Y:=L_2(\Xi;H^1_0(\Omega))$ is equipped with $\|\bigcdot\|_Y:=\|\nabla_x \bigcdot\|_{L_2(\Sigma)^d}$, where $\Sigma:=\Xi \times \Omega$. 

Let for a.e. $t \in \Xi$, $A(t)$ be a quasi-linear elliptic spatial operator of second-order of the form
$$
(A(t) \eta)(\xi)=\int_\Omega \mu(t,x,|\nabla \eta(x)|^2) \nabla \eta(x) \bigcdot \nabla \xi(x) \diff x \quad (\eta,\xi \in H^1_0(\Omega)),
$$
where $t \mapsto (A(t) \eta)(\xi)$ is measurable, and for some constants $m_\mu, M_\mu>0$,
$$
m_\mu (r-s) \leq \mu(\bigcdot,\bigcdot,r^2) r-\mu(\bigcdot,\bigcdot,s^2) s \leq M_\mu(r-s)\quad (r \geq s \geq 0).
$$
Then  $A(t) \colon V \rightarrow V'$ for a.e. $t \in \Xi$,  and so $A \colon Y \rightarrow Y'$, are Lipschitz continuous and strongly monotone
 with constants $L_A=3 M_\mu$ and $m_A=m_\mu$ \cite[Lemma 25.26]{318.3}. 
%\rs{Dit lijkt niet optimaal, want voor $\mu \equiv 1$ geeft dit $m_A=1$ en $L_A=3$ (i.p.v.~$L_A=1$). Ik vind echter nergens scherpere resultaten. In  \cite{45.485} worden -andere- condities voor $\mu$ aangenomen. Die zijn niet standaard, en dat wordt ook niet meer in recentere publicaties van Wihler gedaan. -Evt.- (ik twijfel)~wel mogelijk een term $f$ als in \cite{45.485} toe te voegen.} 
\end{example}

%%%%%%%%%%%%%%
\section{Lipschitz continuous and strongly monotone operators} \label{sec:LipMono}
%%%%%%%%%%%%%%

Since they will be used frequently, we recall some basic facts about Lipschitz continuous and strongly monotone operators.

For a Hilbert space $Z$ (for convenience over $\R$), let $G\colon Z \rightarrow Z'$ be such that for some constants $L_G, m_G>0$, $G$ is \emph{Lipschitz continuous}, i.e.,
\begin{align}  \label{eq:Lip}
\|G w-G v\|_{Z'} & \leq L_G  \|w-v\|_Z \quad (w,v\in Z), 
\intertext{and \emph{strongly monotone}, i.e.,} \label{eq:mono}
(G w-G v)(w-v) &\geq m_G \|w-v\|_Z^2 \quad (w,v \in Z).
\end{align}

Setting $T:=R_Z^{-1}G$, where $R_Z\colon Z \rightarrow Z'$ is the \emph{Riesz map}, from \eqref{eq:Lip}-\eqref{eq:mono} it follows that for $T_\zeta:=\identity -\zeta T$,
$$
\|T_\zeta w-T_\zeta v\|^2_Z \leq (1-2 m_G\zeta+L_G^2 \zeta^2) \|w-v\|_Z^2 \quad (w,v \in Z).
$$
For $\zeta \in (0,\frac{2m_G}{L_G^2})$, we have $1-2 m_G \zeta+L_G^2 \zeta^2 <1$ with minimal value $1-\frac{m_G^2}{L_G^2}$ for $\zeta=\zeta^*:=\frac{m_G}{L_G^2}$.
For $w_0$ arbitrary, the sequence $\{w_i\}$ defined by $w_{i+1}:=T_{\zeta^*} w_i$ converges to the unique fixed point $w$ of $T_{\zeta^*} (w)=w$, being thus the \emph{unique} solution of $G w=0$. This \emph{Picard iteration} is in this setting also known as the \emph{Zarantonello iteration} \cite{317.75}.

If $G$ satisfies \eqref{eq:Lip}-\eqref{eq:mono}, then for any $f \in Z'$ so does $\widetilde{G} w :=G w-f$, with $L_{\widetilde{G}}=L_G$, $m_{\widetilde{G}}=m_G$.
From the fact that thus $\widetilde{G} w=0$ has a unique solution, it follows that $G^{-1}\colon Z' \rightarrow Z$ exists.

By \eqref{eq:mono}, we have $m_G \|w-v\|_Z \leq \|G w-G v\|_{Z'}$, and so
$$
\|G^{-1} f-G^{-1} g\|_Z \leq \tfrac{1}{m_G} \|f-g\|_{Z'} \quad (f,g \in Z'),
$$
i.e., $G^{-1}$ is Lipschitz continuous with $L_{G^{-1}}=\frac{1}{m_G}$.

From \eqref{eq:Lip} and \eqref{eq:mono} we infer
\begin{align*}
\|G w-G v\|_{Z'}^2 &\leq L_G^2 \|w-v\|_Z^2 \leq  \tfrac{L_G^2}{m_G}  (G w-G v)(w-v),
\intertext{and so,} 
(f-g)(G^{-1} f-G^{-1} g) & \geq \tfrac{m_G}{L_G^2} \|f-g\|_{Z'}^2,\quad(f,g \in Z'),
\end{align*}
i.e., $G^{-1}$ is strongly monotone with $m_{G^{-1}}=\tfrac{m_G}{L_G^2}$.

%%%%%%%%%%%%%%
\section{An equivalent formulation of our IVP amenable to Galerkin discretizations} \label{sec:4}
%%%%%%%%%%%%%%

Already because the co-domain $Y' \times H$ of $B_e$ is not the dual of its domain $X$, equation \eqref{eq:33} is not a suitable starting point for creating numerical approximations.
For \emph{linear} $A$, in \cite{11,249.99,249.992} least squares, also known as minimal residual approximations are obtained by replacing $X$ by a subspace in 
$u=\argmin_{w \in X} \|(\ell,u_0)-B_ew\|_{Y' \times H}^2$, and by simultaneously replacing the dual norm on $Y'$ by a computable discretized dual norm.
Although for linear $A$ connected to the method that will be developed here (see \cite[Rem.~4.1]{249.992}), 
for \emph{nonlinear} $A$ the minimal residual approach does not lead to an easily manageable discretization that is necessarily `stable'.
Instead, with 
$$
\NN:=\left[\begin{array}{@{}cc@{}} A & \diff_t\\ \diff_t'& -(A+\gamma_T' \gamma_T)\end{array}\right]\colon Y \times X \rightarrow Y' \times X',
$$
we consider the system 
\be \label{eq:BEN}
\NN
\left[\begin{array}{@{}c@{}} \theta \\ u \end{array}\right]=
\left[\begin{array}{@{}c@{}} \ell \\ -(\ell+\gamma_0' u_0) \end{array}\right].
\ee
As we will see in Corollary~\ref{corol:stability}, $N$ is invertible. By applying \eqref{eq:34}, one directly verifies the following result.

\begin{theorem} \label{thm:basic} The unique solution of \eqref{eq:BEN} is given by $\theta=u:=B_e^{-1}(\ell,u_0)$.
\end{theorem}

Although it will not be used in this work, in the next remark it is shown how, under an additional condition, \eqref{eq:BEN} can be derived from the application of the so-called Br\'{e}zis-Ekeland-Nayroles variational principle (\cite{35.81,233.5}).

\begin{remark} Suppose $A$ has a proper potential $\varphi\colon Y \rightarrow \R \cup \{\infty\}$, i.e., the G\^{a}teaux derivative $d\varphi$ equals $A$. Then it is known (\cite{15.35,75.374}) that the solution $u$ of \eqref{eq:33} is given by $u=\argmin_{w \in X} I(w)$, where, with $\varphi^*\colon Y' \rightarrow \R \cup\{\infty\}$ denoting the Fenchel-Legendre dual of $\varphi$, it holds that
$$
I(w):=\varphi(w)+\varphi^*(\ell-\diff_t w) - (\ell-\diff_t w)(w)+\tfrac12\|\gamma_0 w-u_0\|_H^2.
$$
This $I$ is G\^{a}teaux differentiable at any $w \in X$, with
$$
\diff I(w)=A w+\gamma_T'\gamma_T w-\ell-\gamma_0' u_0-\diff_t' A^{-1}(\ell-\diff_t w),
$$
as one may verify by direct calculation.
As we will see in Lemma~\ref{lem:Schur}, $\diff I\colon X \rightarrow X'$ (there denoted by $S_{\ell,-(\ell+\gamma_0' u_0)}$) is Lipschitz continuous and strongly monotone, so that $\diff I(u)=0$ has a unique solution.
As a consequence of the strong monotonicity of $\diff I$, $I$ is strictly convex, which shows that the solution of $\diff I(u)=0$ is the unique minimizer of $I$.
By setting $\theta:=A^{-1}(\ell-\diff_tu)$, the equation $\diff I(u)=0$ is just the system \eqref{eq:BEN}.

Knowing that the solution $u$ of \eqref{eq:BEN} solves \eqref{eq:33}, using \eqref{eq:34} one verifies that $I(u)=0$.
\end{remark}

\begin{remark}
For linear $A$ and a homogeneous initial condition, \eqref{eq:BEN} has a similar structure as a preconditioned version  of the (linear) parabolic PDE discretized by implicit Euler presented in \cite{234.7}.
\end{remark}

We will study (families of) Galerkin discretizations of \eqref{eq:BEN}. 
Let $(X^\delta)_{\delta \in \Delta}$ and $(Y^\delta)_{\delta \in \Delta}$ be families of closed non-trivial subspaces of $X$ and $Y$, respectively.
Sometimes we will refer to $X^\delta$ and $Y^\delta$ as the \emph{trial} and \emph{test} spaces.
We include a $\delta_* \in \Delta$ for which $X^{\delta_*}=X$ and $Y^{\delta_*}=Y$, which has the following consequence:\vspace*{.5ex}
\begin{center}
\emph{All results that will be shown for $\delta \in \Delta$ include the continuous setting.}
\end{center}
We will assume that $X^\delta$ and $Y^\delta$ for $\delta \neq \delta_*$  are finite dimensional. The trivial embeddings $X^\delta \hookrightarrow X$ and $Y^\delta \hookrightarrow Y$ will be denoted by $E_X^\delta$ and $E_Y^\delta$, respectively, which operators as well as their duals we often write for clarity.
%Dus is voor iedere $\delta \neq \delta_*$,  is $\bigcdot\|_{X^\delta}$ equivalent met $\|E_X^\delta \bigcdot \|_$, en dus is nderdaad de inbedding van $X^\delta$ in $X$ continue}}
\smallskip

We equip $Y^\delta$ with its natural (Hilbertian) norm
\begin{align*}
&\|\bigcdot\|_{Y^\delta}:=\|E_Y^\delta \bigcdot\|_Y,
\intertext{so that}
&\|\bigcdot\|_{(Y^\delta)'}=\sup_{0 \neq v \in Y^\delta}\frac{\bigcdot(v)}{\|v\|_{Y^\delta}}.
\intertext{On $X$, besides $\|\bigcdot\|_X$ we define additional, and, concerning the measurement of the temporal derivative, \emph{`mesh'-dependent} norms,}
&\|\bigcdot\|_{X,\delta}:=\sqrt{\| \bigcdot\|^2_{Y}+\|(E_Y^\delta)' \diff_t \bigcdot\|^2_{(Y^\delta)'}+\|\gamma_T \bigcdot\|_H^2}\,,
\intertext{and equip $X^\delta$ with}
&\|\bigcdot\|_{X^\delta}:=\|E_X^\delta \bigcdot\|_{X,\delta}\,,
\intertext{so that}
&\|\bigcdot\|_{(X^\delta)'}=\sup_{0 \neq v \in X^\delta}\frac{\bigcdot(v)}{\|v\|_{X^\delta}}.
\end{align*}
We emphasize that for $\delta\neq \delta_*$, the norm on $X^\delta$, and therefore that on $(X^\delta)'$, depend on the choice of the subspace $Y^\delta \subset Y$ that will be clear from the context.

The above norms $\|\bigcdot\|_{Y^{\delta_*}}$, $\|\bigcdot\|_{(Y^{\delta_*})'}$, $\|\bigcdot\|_{X^{\delta_*}}$, $\|\bigcdot\|_{(X^{\delta_*})'}$ will usually be written without the index $\delta_*$. Thanks to \eqref{eq:embedding}, $\|\bigcdot\|_X=\sqrt{\|\bigcdot\|^2_{Y}+\|\diff_t \bigcdot\|^2_{Y'}+\|\gamma_T \bigcdot\|_H^2}$ is equivalent to the canonical norm $\sqrt{\|\bigcdot\|^2_{Y}+\|\diff_t \bigcdot\|^2_{Y'}}$ on $X$.
\smallskip

The following result extends the `inf-sup identity' known for $\delta=\delta_*$ (e.g.~\cite{70.95}), to general $\delta \in \Delta$.
For convenience we recall its short proof.

\begin{lemma}[{\cite[Lemma 3.4]{249.99}}] With the \emph{Riesz map} $R_Y\colon Y \rightarrow Y'$, it holds that
$$
\|\bigcdot\|_{X,\delta}^2=\|(\diff_t+R_Y)\bigcdot\|_{(Y^\delta)'}^2+\|\gamma_0 \bigcdot\|_H^2 \qquad\text{on } X \cap Y^\delta.
$$
\end{lemma}

\begin{proof} For $z \in X \cap Y^\delta$, let $y \in Y^\delta$ be defined by $(R_Y y)(v)=(\diff_t z)(v)$ ($v \in Y^\delta$). Then $(R_Y y)(y)=\|\diff_t z\|_{(Y^\delta)'}^2$.
Furthermore, for $v \in Y^\delta$ we have $((\diff_t+R_Y)z)(v)=(R_Y(y+z))(v)$, and so, thanks to $z \in Y^\delta$,
\begin{align*}
\|(\diff_t+R_Y)z\|_{(Y^\delta)'}^2&=(R_Y(y+z))(y+z)=(R_Y y)(y)+2(R_Y y)(z)+(R_Y z)(z)\\
&=(R_Y y)(y)+2 (\diff_t z)(z)+(R_Y z)(z)=\|z\|_{X,\delta}^2-\|\gamma_0 z\|_H^2,
\end{align*}
where we used  $\diff_t+\diff_t'=\gamma_T' \gamma_T-\gamma_0' \gamma_0$ as mappings $X \rightarrow X'$.
\end{proof}

Our only application of the above lemma will be in the following immediate corollary. In view of \eqref{eq:embedding}, it is mainly relevant for $\delta \neq \delta_*$.

\begin{corollary} \label{corol:2} 
For $X^\delta \subset Y^\delta$, it holds that $
\|\gamma_0 E_X^\delta\|_{X^\delta \rightarrow H} \leq 1$.
\end{corollary}

With
$$
A_Y^\delta:=(E_Y^\delta)' A E_Y^\delta,\quad
A_X^\delta:=(E_X^\delta)' A E_X^\delta,\quad
\diff_t^\delta:=(E_Y^\delta)' \diff_t E_X^\delta,\quad
\gamma_{t}^\delta:=\gamma_{t} E_X^\delta,
$$
we set
$$
\NN^\delta := \left[\begin{array}{@{}cc@{}} A_Y^\delta & \diff_t^\delta\\ (\diff_t^\delta)'& -(A_X^\delta+(\gamma^\delta_T)' \gamma^\delta_T)\end{array}\right]\colon Y^\delta \times X^\delta \rightarrow (Y^\delta)' \times (X^\delta)',
$$
being the operator resulting from the \emph{Galerkin discretization} of \eqref{eq:BEN}.

\begin{proposition} \label{prop:Lip} $N^\delta\colon Y^\delta \times X^\delta \rightarrow (Y^\delta)' \times (X^\delta)'$ is \emph{Lipschitz continuous}, with constant $L_{\NN^\delta}=L_\NN:=L_A+1$.
\end{proposition}

\begin{proof} For $(\xi,z), (\mu,w), (\sigma,v) \in Y^\delta \times X^\delta$,
\begin{align*}
&\Big(\NN^\delta
\left[\begin{array}{@{}c@{}} \xi \\ z \end{array}\right]
-
\NN^\delta
\left[\begin{array}{@{}c@{}} \mu \\ w\end{array}\right]
\Big)
\left[\begin{array}{@{}c@{}} \sigma \\ v \end{array}\right]
=
(A_Y^\delta \xi)(\sigma)-(A_Y^\delta \mu)(\sigma)+(\diff_t^\delta(z-w))(\sigma)\\
&\hspace*{1cm}+(\diff_t^\delta v)(\xi-\mu)+(A_X^\delta w)(v)-(A_X^\delta z)(v)+\langle \gamma_T^\delta (w-z),\gamma_T^\delta v\rangle_H\\
&\leq L_A\|\xi-\mu\|_{Y^\delta} \|\sigma\|_{Y^\delta}+\|\diff_t^\delta(z-w)\|_{(Y^\delta)'}\|\sigma\|_{Y^\delta}+\|\xi-\mu\|_{Y^\delta}\|\diff_t^\delta v\|_{(Y^\delta)'}\\
&\hspace{1em} + L_A\|E_X^\delta(z-w)\|_Y \|E_X^\delta v\|_Y+\|\gamma_T^\delta (z-w)\|_H \|\gamma_T^\delta v\|_H\\
& \leq
\Big[ \big((1+L_A) \|\xi-\mu\|_{Y^\delta}^2+\max(1,L_A)\|z-w\|_{X^\delta}^2\big)\\
&\hspace{6em} \times \big((1+L_A) \|\sigma\|_{Y^\delta}^2+\max(1,L_A)\|v\|_{X^\delta}^2\big)\Big]^\frac12\\
& \leq (1+L_A) \big(\|z-w\|_{X^\delta}+\|\xi-\mu\|_{Y^\delta}\big)(\|v\|_{X^\delta}+\|\sigma\|_{Y^\delta}\big),
\end{align*}
which completes the proof.
 \end{proof}
 
 For $(f^\delta,g^\delta) \in (Y^\delta)' \times (X^\delta)'$,  consider the problem of finding $(\theta^\delta,u^\delta) \in Y^\delta \times X^\delta$ such that
\be \label{eq:Galsystem}
\NN^\delta \left[\begin{array}{@{}c@{}} \theta^\delta \\ u^\delta \end{array}\right]=\left[\begin{array}{@{}c@{}} f^\delta \\ g^\delta \end{array}\right].
\ee
From $A\colon Y \rightarrow Y'$ being Lipschitz continuous and strongly monotone, it follows that so is $A_Y^\delta\colon Y^\delta \rightarrow (Y^\delta)'$, with constants as for $A$. In particular, $A_Y^\delta$ is invertible. 
By eliminating $\theta^\delta$ from \eqref{eq:Galsystem} we conclude the following result.

\begin{proposition} \label{prop:equivSchur} With $S^\delta=S_{f^\delta,g^\delta}^\delta\colon X^\delta \rightarrow (X^\delta)'$ defined by
$$
S^\delta z:=A_X^\delta z+(\gamma_T^\delta)'\gamma_T^\delta z+g^\delta-(\diff_t^\delta)' (A_Y^\delta)^{-1}(f^\delta-\diff_t^\delta z),
$$
\eqref{eq:Galsystem} is equivalent to
$$
S^\delta u^\delta=0,
$$
and $\theta^\delta:=(A_Y^\delta)^{-1}(f^\delta-\diff_t^\delta u^\delta)$.
\end{proposition}

\begin{lemma} \label{lem:Schur} $S^\delta:X^\delta \rightarrow (X^\delta)'$ is Lipschitz continuous with constant $L_{S^\delta}=L_{S}:=\max\big(1,L_A,\tfrac{1}{m_A} \big)$, and strongly monotone with constant $m_{S^\delta}=m_{S}=\min\big(1,m_A,\tfrac{m_A}{L_A^2}\big)$.
\end{lemma}

\begin{proof} For $w,z, v \in X^\delta$, we have
\begin{align*}
\big(S^\delta&w-S^\delta z\big)(v)=\big(A_X^\delta w-A_X^\delta z\big)(v)+\langle \gamma_T^\delta(w-z),\gamma_T^\delta v\rangle_H\\
&\hspace*{9em}  +(\diff_t^\delta v)\big((A_Y^\delta)^{-1}(f^\delta-\diff^\delta_t z)-(A_Y^\delta)^{-1}(f^\delta-\diff^\delta_t w)\big)\\
&\leq L_A \|E_X^\delta(w-z)\|_Y \|E_X^\delta v\|_Y+\| \gamma_T^\delta(w-z)\|_H \|\gamma_T^\delta v\|_H\\
&\hspace*{9em} +\|\diff_t^\delta v\|_{(Y^\delta)'} \tfrac{1}{m_A} \|\diff^\delta_t(w-z)\|_{(Y^\delta)'} \\
& \leq \max\big(1,L_A,\tfrac{1}{m_A} \big)\|w-z\|_{X^\delta} \|v\|_{X^\delta},
\end{align*}
and for $w,z \in X^\delta$, we have
\begin{align*}
\big(S^\delta w&-S^\delta z\big)(w-z) \\
&\geq m_A\|E_X^\delta(w-z)\|_Y^2+\|\gamma_T^\delta(w-z)\|_H^2+\tfrac{m_A}{L_A^2}\|\diff_t^\delta(w-z)\|_{(Y^\delta)'}^2\\
& \geq \min\big(1,m_A,\tfrac{m_A}{L_A^2}\big)\|w-z\|_{X^\delta}^2. \qedhere
\end{align*}
\end{proof}
 
\begin{corollary} \label{corol:stability} $\NN^\delta\colon Y^\delta \times X^\delta \rightarrow (Y^\delta)' \times (X^\delta)'$ is invertible, with $(\NN^\delta)^{-1}$  being \emph{Lipschitz continuous} with constant
 $$
 L_{(\NN^\delta)^{-1}}= L_{\NN^{-1}}:=\tfrac{1}{m_A}+\max\big(1,\tfrac{1}{m_A})\big(\tfrac{1}{m_S}(1+\tfrac{1}{m_A})\big).
 $$ 
\end{corollary}

\begin{proof} For $(f^\delta,g^\delta), (\tilde{f}^\delta,\tilde{g}^\delta) \in (Y^\delta)' \times (X^\delta)'$,  consider
$$
\NN^\delta \left[\begin{array}{@{}c@{}} \theta^\delta \\ u^\delta \end{array}\right]=\left[\begin{array}{@{}c@{}} f^\delta \\ g^\delta \end{array}\right], \quad
\NN^\delta \left[\begin{array}{@{}c@{}} \tilde{\theta}^\delta \\ \tilde{u}^\delta \end{array}\right]=\left[\begin{array}{@{}c@{}} \tilde{f}^\delta \\ \tilde{g}^\delta \end{array}\right].
$$
or, equivalently,
$$
S^\delta u^\delta=0,\, \theta^\delta=(A_Y^\delta)^{-1}(f^\delta-\diff_t^\delta u^\delta),\quad
\tilde{S}^\delta \tilde{u}^\delta=0,\, \tilde{\theta}^\delta=(A_Y^\delta)^{-1}(\tilde{f}^\delta-\diff_t^\delta \tilde{u}^\delta)
$$
where $\tilde{S}^\delta:=S_{\tilde{f}^\delta,\tilde{g}^\delta}^\delta$.

Using that $\|(\diff^\delta_t)'\|_{Y^\delta \rightarrow (X^\delta)'} \leq 1$, and $(A_Y^\delta)^{-1}\colon (Y^\delta)' \rightarrow Y^\delta$ is Lipschitz continuous with constant $\frac{1}{m_A}$, one  infers that
\begin{align*}
\|S^\delta \widetilde{u}^\delta\|_{(X^\delta)'}&=\|S^\delta \widetilde{u}^\delta-\widetilde{S}^\delta \widetilde{u}^\delta\|_{(X^\delta)'}\\ & \leq
\|g^\delta-\widetilde{g}^\delta\|_{(X^\delta)'}+\|(A_Y^\delta)^{-1}(\widetilde{f}^\delta-\diff_t^\delta \widetilde{u}^\delta)-(A_Y^\delta)^{-1}(f^\delta-\diff_t^\delta \widetilde{u}^\delta)\|_{Y^\delta}
\\ & \leq \|g^\delta-\widetilde{g}^\delta\|_{(X^\delta)'}+\tfrac{1}{m_A} \|f^\delta-\widetilde{f}^\delta\|_{(Y^\delta)'},
\end{align*}
and so
\begin{align*}
m_{S} \|\widetilde{u}^\delta-u^\delta\|_{X^\delta}^2 &\leq \big(S^\delta \widetilde{u}^\delta-S^\delta u^\delta \big)(\widetilde{u}^\delta-u^\delta) \\
&\leq
\big(\|g^\delta-\widetilde{g}^\delta\|_{(X^\delta)'}+\tfrac{1}{m_A}\|f^\delta-\widetilde{f}^\delta\|_{(Y^\delta)'}\big)\|\widetilde{u}^\delta-u^\delta\|_{X^\delta},
\end{align*}
i.e.,
\be \label{eq:u}
\|\widetilde{u}^\delta-u^\delta\|_{X^\delta} \leq \tfrac{1}{m_{S}}\big(\|g^\delta-\widetilde{g}^\delta\|_{(X^\delta)'}+\tfrac{1}{m_A}\|f^\delta-\widetilde{f}^\delta\|_{(Y^\delta)'}\big).
\ee
Similarly, we infer that
\begin{align} \nonumber
\|\theta^\delta-\widetilde{\theta}^\delta\|_{Y^\delta} &\leq \tfrac{1}{m_A} \big(\|f^\delta-\widetilde{f}^\delta\|_{(Y^\delta)'}+\|\diff_t^\delta (u^\delta-\widetilde{u}^\delta)\|_{(Y^\delta)'}\big)\\ \label{eq:lambda}
&\leq \tfrac{1}{m_A} \big(\|f^\delta-\widetilde{f}^\delta\|_{(Y^\delta)'}+\|u^\delta-\widetilde{u}^\delta\|_{X^\delta}).
\end{align}
The combination of \eqref{eq:u} and \eqref{eq:lambda} completes the proof.
\end{proof}

As has been already noticed earlier, invertibility of $B_e\colon X \rightarrow Y' \times H$ is well-known, even for a larger class of operators $A\colon Y \rightarrow Y'$ than the Lipschitz continuous and strongly monotone operators that we consider. For this class of operators $A$, however, we could not find the following fundamental result about the Lipschitz continuous dependency of the solution $u$ of $B_e u= (\ell, u_0)$ on the data.

\begin{corollary} \label{corol:Lipstab} $B_e^{-1}\colon Y' \times H \rightarrow X$ is \emph{Lipschitz continuous} with constant $L_{B_e^{-1}}=\frac{1}{m_S}(1+\frac{1}{m_A})$.
\end{corollary}

\begin{proof} Let $B_e u=(\ell,u_0)$, $B_e \widetilde{u}=(\widetilde{\ell},\widetilde{u}_0)$.
Theorem~\ref{thm:basic} and Proposition~\ref{prop:equivSchur} show that $S_{\ell,-(\ell+\gamma_0' u_0)} u=0$ and $S_{\widetilde{\ell},-(\widetilde{\ell}+\gamma_0' \widetilde{u}_0)} \widetilde{u}=0$. Using that $\|\bigcdot\|_{X'} \leq \|\bigcdot\|_{Y'}$, and, as a consequence of Corollary~\ref{corol:2}, $\|\gamma_0'\|_{H \rightarrow X'}\leq 1$, equation \eqref{eq:u} shows that 
$$
\|\widetilde{u}-u\|_X \leq \tfrac{1}{m_S}((1+\tfrac{1}{m_A})\|\widetilde{\ell}-\ell\|_{Y'}+\|\widetilde{u}_0-u_0\|_H),
$$
which completes the proof.
\end{proof}

\begin{remark}[Time periodic problems] Instead of with the initial condition $u(0)=u_0$, the first order evolution problem~\eqref{eq:IVP} can also be completed by the condition $u(t+T)=u(t)$. Applications are found in e.g.~electric engineering \cite{70.32}. By taking $X:=L_2(\Xi;V) \cap H^1_{\text{per}}(\Xi;V')$, $\|\bigcdot\|_{X,\delta}:=\sqrt{\| \bigcdot\|^2_{Y}+\|(E_Y^\delta)' \diff_t \bigcdot\|^2_{(Y^\delta)'}}$\,, $B_e:=B$, $\NN:=\left[\begin{array}{@{}cc@{}} A & \diff_t\\ \diff_t'& -A\end{array}\right]\colon Y \times X \rightarrow Y' \times X'
$, and
$\NN
\left[\begin{array}{@{}c@{}} \theta \\ u \end{array}\right]=
\left[\begin{array}{@{}c@{}} \ell \\ -\ell \end{array}\right]$, using that $\diff_t+\diff_t'=0$ as mapping $X \rightarrow X'$ the preceding (and subsequent) analysis easily generalizes to this case.
\end{remark}
%%%%%%%%%%%%%%
\section{Quasi-optimal discretizations} \label{sec:5}
%%%%%%%%%%%%%%

Below, in \eqref{eq:LBB}, we impose a uniform `inf-sup' or Ladyzhenskaya–Babu\v{s}ka-Brezzi (LBB) condition \eqref{eq:LBB}, under which the norms $\|\bigcdot\|_{X^\delta}$ and $\|E_X^\delta \bigcdot\|_{X}$ on $X^\delta$ are \emph{uniformly equivalent}. 
For $(f,g) \in Y' \times X'$, considering the systems
\be \label{eq:36}
\NN \left[\begin{array}{@{}c@{}} \theta \\ u \end{array}\right]=\left[\begin{array}{@{}c@{}} f \\ g \end{array}\right],\quad
\NN^\delta \left[\begin{array}{@{}c@{}} \theta^\delta \\ u^\delta \end{array}\right]=\left[\begin{array}{@{}c@{}} (E_Y^\delta)' f \\ (E_X^\delta)' g \end{array}\right],
\ee
we show that under this condition the Galerkin approximation $(\theta^\delta,u^\delta)$ for $(\theta,u)$ is \emph{quasi-optimal}.

\begin{proposition}[Quasi-optimal Galerkin approximation] \label{prop:9} Assume that
\be \label{eq:LBB}
\gamma_\Delta=\gamma_\Delta\big((X^\delta,Y^\delta)_{\delta \in \Delta}\big):=\inf_{\delta \in \Delta} \inf_{\{z \in X^\delta\colon \diff_t z \neq 0\}} \frac{\|\diff_t^\delta z \|_{(Y^\delta)'}}{\|\diff_t  E_X^\delta z\|_{Y'}}>0.
\ee
Then for the solutions of the systems in \eqref{eq:36}, we have
\begin{align*}
\|\theta-E_Y^\delta \theta^\delta\|_Y+\|u&-E_X^\delta u^\delta\|_X 
\\ &\leq \big(1+\tfrac{L_{N^{-1}} L_N}{\gamma_\Delta^2}\big)
\inf_{(\bar{\theta}^\delta,\bar{u}^\delta) \in Y^\delta \times X^\delta} 
\|\theta-E_Y^\delta \bar{\theta}^\delta\|_Y+\|u-E_X^\delta \bar{u}^\delta\|_X.
\end{align*}
If for some $(\ell,u_0) \in Y' \times H$, $(f,g)=(\ell,-(\ell+\gamma_0' u_0))$ (i.e.~$B_e u=(\ell, u_0)$), and $X^\delta \subset Y^\delta$, then
$$
\|u-E_X^\delta u^\delta\|_X \leq 2 \big(1+\tfrac{L_{N^{-1}} L_N}{\gamma_\Delta^2}\big) \inf_{\bar{u}^\delta \in  X^\delta} 
\|u-E_X^\delta \bar{u}^\delta\|_X.
$$
\end{proposition}

\begin{proof}
With the shorthand notations $\vec{E}^\delta:=(E_Y^\delta,E_X^\delta)$, $\vec{u}^\delta:=(\theta^\delta,u^\delta)$, and $\vec{u}:=(\theta,u)$, we have
$\|\vec{E}^\delta\|_{Y^\delta \times X^\delta \rightarrow Y \times X}=
\|(\vec{E}^\delta)'\|_{Y' \times X' \rightarrow (Y^\delta)' \times (X^\delta)' } \leq \frac{1}{\gamma_\Delta}$, and
 the Galerkin property $(\vec{E}^\delta)' N \vec{u}=N^\delta \vec{u}^\delta$. For arbitrary $\vec{\underline{u}}^\delta \in Y^\delta \times X^\delta$,
we estimate
$$
\|\vec{u}-\vec{E}^\delta \vec{u}^\delta\|_{Y \times X} \leq \|\vec{u}-\vec{E}^\delta \vec{\underline{u}}^\delta\|_{Y \times X}+  \tfrac{1}{\gamma_\Delta} \|\vec{\underline{u}}^\delta-\vec{u}^\delta\|_{Y^\delta \times X^\delta},
$$
and, using Corollary~\ref{corol:stability} and Proposition~\ref{prop:Lip},
\begin{align*}
\|\vec{\underline{u}}^\delta-\vec{u}^\delta\|_{Y^\delta \times X^\delta} &\leq L_{N^{-1}}\|N^\delta \vec{\underline{u}}^\delta -N^\delta \vec{u}^\delta\|_{(Y^\delta)' \times (X^\delta)'}\\
&= L_{N^{-1}}\|(\vec{E}^\delta)' (N \vec{E}^\delta \vec{\underline{u}}^\delta -N \vec{E}^\delta\vec{u}^\delta)\|_{(Y^\delta)' \times (X^\delta)'} \\
&= L_{N^{-1}}\|(\vec{E}^\delta)' (N \vec{E}^\delta \vec{\underline{u}}^\delta -N \vec{u})\|_{(Y^\delta)' \times (X^\delta)'}\\
& \leq L_{N^{-1}} \tfrac{1}{\gamma_\Delta}\| N \vec{E}^\delta \vec{\underline{u}}^\delta -N \vec{u}\|_{Y' \times X'}\\
& \leq L_{N^{-1}} L_N  \tfrac{1}{\gamma_\Delta}\|\vec{E}^\delta \vec{\underline{u}}^\delta - \vec{u}\|_{Y\times X},
\end{align*}
so that
\begin{align*}
\|\theta-E_Y^\delta \theta^\delta\|_Y+\|u&-E_X^\delta u^\delta\|_X 
\\ &\leq \big(1+\tfrac{L_{N^{-1}} L_N}{\gamma_\Delta^2}\big)
\inf_{(\bar{\theta}^\delta,\bar{u}^\delta) \in Y^\delta \times X^\delta} 
\|\theta-E_Y^\delta \bar{\theta}^\delta\|_Y+\|u-E_Y^\delta \bar{u}^\delta\|_X.
\end{align*}
The final statement is a consequence of Theorem~\ref{thm:basic} and $X^\delta \subset Y^\delta$.
\end{proof}

Notice that the inf-sup condition \eqref{eq:LBB} depends on $(X^\delta,Y^\delta)_{\delta \in \Delta}$, but not on the operator $A$.
It is the same condition as studied in previous works about minimal residual discretizations for the case of a linear $A$.

\subsection{`Inf-sup' stability condition \eqref{eq:LBB}}
We briefly summarize results from \cite{11,249.99,249.992,249.991} about situations for which validity of \eqref{eq:LBB} with $\dim Y^\delta \lesssim \dim X^\delta$ has been verified. Concerning this last condition, \eqref{eq:LBB} holds obviously true with $\gamma_\Delta=1$ when $Y^\delta=Y$, whilst $\dim Y^\delta \lesssim \dim X^\delta$ means that the dimension of the Galerkin system to be solved is of the order of the dimension of the trial space $X^\delta$.
\begin{example}[Tensor product setting] \label{ex:tensor}
For $\delta \in \Delta \setminus \delta_*$, let 
$$
X^\delta = X_t^\delta \otimes X_x^\delta,\quad Y^\delta = Y_t^\delta \otimes Y_x^\delta.
$$
Then a tensor product argument shows that with
\be \label{eq:tensor}
\gamma_{\Delta,t}:=\inf_{\delta \in \Delta} \! \inf_{\{z_t \in X_t^\delta\colon \tfrac{\diff z_t}{\diff t} \neq 0\}}\hspace*{-1.2em}\frac{\sup_{0 \neq v_t \in Y_t^\delta} \frac{\int_\Xi \tfrac{\diff z_t}{\diff t} v_t\diff t}{{\|v_t\|_{L_2(\Xi)}}}}{\|\tfrac{\diff z_t}{\diff t}\|_{L_2\Xi)}}, \,\,
\gamma_{\Delta,x}:=\inf_{\delta \in \Delta}\inf_{0 \neq z_x \in X_x^\delta}\hspace*{-0.8em}\frac{\sup_{0 \neq v_x \in Y_x^\delta} \frac{\langle z_x, v_x\rangle_H}{\|v_x\|_V}}{\|z_x\|_{V'}},
\ee
it holds that $\gamma_{\Delta} \geq \gamma_{\Delta,t} \gamma_{\Delta,x}$.
Clearly, $\gamma_{\Delta,t} >0$ if $\frac{\diff}{\diff t} X_t^\delta \subset Y_t^\delta$ (which is, however, not a necessary condition (see \cite[Prop.~6.1]{11} for an example)).

If $Y_x^\delta=X_x^\delta$, then $\gamma_{\Delta,x}=\sup_{\delta \in \Delta} \|P^\delta\|_{V \rightarrow V}^{-1}$ where $P^\delta$ is the $H$-orthogonal projector onto $X_x^\delta$.
For $(H,V)=(L_2(\Omega),H^1_0(\Omega))$ for some domain $\Omega \subset \R^d$, uniform boundedness in $H^1(\Omega)$-norm of the $L_2(\Omega)$-orthogonal projector onto finite element spaces $X_x^\delta$ is an intensively studied subject.
Recently in \cite{64.596} it was shown that for $d<7$ this uniform boundedness holds true for the family of all Lagrange finite element spaces of arbitrary degree w.r.t.~all conforming newest vertex bisection meshes that can be created from an initial mesh.
\end{example}

Above example shows that \eqref{eq:LBB} with $\dim Y^\delta \lesssim \dim X^\delta$ can be realized for trial and test spaces being finite element spaces w.r.t.~partitions of $\Sigma=\Xi \times \Omega$ that are products of a partition of $\Xi$ and a partition of $\Omega$. The next example shows that this setting can be generalized to the situation that in different time-intervals different spatial partitions are employed.

\begin{example}[Time-slab setting]
Let $(\bar{X}^\delta, \bar{Y}^\delta)_{\delta \in \bar{\Delta}}$ be a family of pairs of closed subspaces of $X$ and $Y$
for which $\gamma_{\bar{\Delta}}=\gamma_{\bar{\Delta}} \big((\bar{Y}^\delta \times \bar{X}^\delta)_{\delta \in \bar{\Delta}}\big)>0$.
Then if, for $\delta \in \Delta$, $X^\delta$ and $Y^\delta$ are such that for some
finite partition $\Xi^\delta=([t_{i-1}^\delta,t_{i}^\delta])_i$ of $\Xi$, with
$G^\delta_i(t):=t_{i-1}^\delta+\frac{t}{T}(t_{i}^\delta-t_{i-1}^\delta)$ and arbitrary $\delta_i \in \bar{\Delta}$ it holds that
\begin{align*}
X^\delta \subseteq \{u \in C(\overline{\Xi};V) &\colon u|_{(t^\delta_{i-1},t^\delta_i)} \circ G^\delta_i \in \bar{X}^{\delta_i}\},\\
Y^\delta \supseteq \{v \in L_2(\Xi;V) &\colon v|_{(t^\delta_{i-1},t^\delta_i)} \circ G^\delta_i \in \bar{Y}^{\delta_i}\},
\end{align*}
then $\gamma^\delta \geq \gamma_{\bar{\Delta}}>0$. 
\end{example}

\begin{remark} Another generalization outside the field of finite elements of the `full' tensor product setting from Example~\ref{eq:tensor} is that to `sparse' tensor products, see \cite{11}.
With the application of \emph{wavelets-in-time} instead of finite elements, this `sparse' setting has been generalized by allowing (adaptive)  `refinements' (more accurately called `enrichments') that are simultaneously localized in time \emph{and} space, see \cite{249.991}. 
\end{remark}

%%%%%%%%%%%%%%
\section{An a posteriori verifiable condition for quasi-optimality} \label{sec:Pjotr}
%%%%%%%%%%%%%%
For $(\ell,u_0) \in Y' \times H$, and $Y^\delta \times X^\delta \subset Y \times X$, recall the equations 
$$
B_e u=(\ell,u_0),\quad \NN^\delta
\left[\begin{array}{@{}c@{}} \theta^\delta \\ u^\delta \end{array}\right]=
\left[\begin{array}{@{}c@{}} (E_Y^\delta)'\ell \\ -(E_X^\delta)'(\ell+\gamma_0' u_0) \end{array}\right],
$$
or equivalently,
$$
Su=S_{\ell,-(\ell+\gamma_0' u_0)} u=0,\quad S^\delta u^\delta=S^\delta_{(E_Y^\delta)'\ell,-(E_X^\delta)'(\ell+\gamma_0' u_0)} u^\delta=0.
$$

As we have seen in Proposition~\ref{prop:9}, under the uniform inf-sup condition \eqref{eq:LBB}, 
$E_X^\delta u^\delta$ is, w.r.t.~the norm on $X$, a quasi-best approximation to $u$. 

Since $u$ may have singularities that are simultaneously localized in time and space, we would like to consider families of finite element partitions that contain partitions that are locally refined simultaneously in time and space. Such partitions, however, do not allow a subdivision into time-slabs.
As will be shown in \cite[Appendix A]{75.293}, with such partitions \eqref{eq:LBB} \emph{cannot} be expected to be valid.
 For that reason in the current section we will investigate an approach that circumvents condition \eqref{eq:LBB}.
We will derive an \emph{a posteriori} verifiable condition \eqref{eq:pjotr} for quasi-optimality.  Whereas  \eqref{eq:LBB} guarantees quasi-optimality \emph{uniformly in the data} $(\ell,u_0) \in Y' \times H$, \eqref{eq:pjotr} will only ensure quasi-optimality for the \emph{data at hand}. The potential benefit is that the latter kind of quasi-optimality might be realizable for smaller $Y^\delta$.

As a preparation, first we show that  $\|u-E_X^\delta u^\delta\|_{X,\delta} \lesssim \inf_{\bar{u}^\delta \in X^\delta} \|u-E_X^\delta \bar{u}^\delta\|_{X}$ whenever $X^\delta \subset Y^\delta$.

\begin{theorem} \label{thm:quasi-opt} With $C_1:=1+\tfrac{1}{m_S}\Big(1+\sqrt{(1+L_A^2)(1 + \frac{1}{m_A^2})}\, \Big)$, 
for $X^\delta \subset Y^\delta$ it holds that
$$
\max\big(\|u-E_X^\delta u^\delta\|_{X,\delta},\|u_0-\gamma_0 E_X^\delta u^\delta\|_H\big)  \leq C_1 \inf_{\bar{u}^\delta \in X^\delta} \|u-E_X^\delta \bar{u}^\delta\|_{X}.
$$
\end{theorem}

\begin{proof} For $\bar{u} \in X^\delta$, we have
\begin{align*}
\|u-E_X^\delta u^\delta\|_{X,\delta}  & \leq \|u-E_X^\delta \bar{u}^\delta\|_X+\| \bar{u}^\delta-u^\delta\|_{X^\delta},
\intertext{as also}
\|u_0-\gamma_0 E_X^\delta u^\delta\|_H & \leq  \|\gamma_0 (u-E_X^\delta \bar{u}^\delta)\|_H+\|\gamma_0 E_X^\delta(\bar{u}^\delta-u^\delta)\|_H\\
&\leq  \|u-E_X^\delta \bar{u}^\delta\|_X+\| \bar{u}^\delta-u^\delta\|_{X^\delta}.
\end{align*}
by two applications of Corollary~\ref{corol:2}.

Recall that
\be \label{eq:37}
S^\delta z:=A_X^\delta z+(\gamma_T^\delta)'\gamma_T^\delta z-(E_X^\delta)'(\ell+\gamma_0' u_0)-(\diff_t^\delta)' (A_Y^\delta)^{-1}((E_Y^\delta)'\ell-\diff_t^\delta z),
\ee
where Lemma~\ref{lem:Schur} shows that $S^\delta\colon X^\delta \rightarrow (X^{\delta})'$ is strongly monotone with constant $m_{S^\delta}=m_{S}=\min\big(1,m_A,\tfrac{m_A}{L_A^2}\big)$.
We estimate
$$
m_{S} \| \bar{u}^\delta-u^\delta\|_{X^\delta}^2 \leq (S^\delta \bar{u}^\delta-S^\delta u^\delta)(\bar{u}^\delta-u^\delta)=(S^\delta \bar{u}^\delta)(\bar{u}^\delta-u^\delta).
$$
Substituting $(\ell,u_0)=B_e u$ in the definition of $S^\delta$, and using $\gamma_T' \gamma_T=\diff_t+\diff_t'+\gamma_0' \gamma_0$ as mappings $X \rightarrow X'$, and $X^\delta \subset Y^\delta$, the latter which allows to write $E_X^\delta=(A_Y^\delta)^{-1} (E_Y^\delta)' A E_X^\delta$ for deriving the third equality below, we infer that
\begin{align} \nonumber
(S^\delta \bar{u}^\delta)(\bar{u}^\delta&-u^\delta)=
\big(A E_X^\delta \bar{u}^\delta+\gamma_T '\gamma_T E_X^\delta \bar{u}^\delta-(\diff_t+A)u-\gamma_0' \gamma_0 u \big)\big(E_X^\delta(\bar{u}^\delta-u^\delta)\big)\\ \nonumber
&\hspace{1.3cm} -\big((E_Y^\delta)' \diff_t E_X^\delta(\bar{u}^\delta-u^\delta)\big)\big((A_Y^\delta)^{-1}(E_Y^\delta)'(Au+\diff_t(u- E_X^\delta \bar{u}^\delta))\big) \\ \nonumber
=&\,\big(A E_X^\delta \bar{u}^\delta-Au+\diff_t(E_X^\delta \bar{u}^\delta-u)+\gamma_0' \gamma_0 (E_X^\delta \bar{u}^\delta-u)\big)\big(E_X^\delta(\bar{u}^\delta-u^\delta)\big)\\ \nonumber
&+\big(\diff_t E_X^\delta (\bar{u}^\delta-u^\delta)\big) \big(E_X^\delta \bar{u}^\delta\big)
\\ \nonumber
&-\big((E_Y^\delta)' \diff_t E_X^\delta(\bar{u}^\delta-u^\delta)\big)\Big((A_Y^\delta)^{-1}(E_Y^\delta)'(Au+\diff_t(u- E_X^\delta \bar{u}^\delta))\big) \\
=&\,\big(A E_X^\delta \bar{u}^\delta-Au+\diff_t(E_X^\delta \bar{u}^\delta-u)+\gamma_0' \gamma_0 (E_X^\delta \bar{u}^\delta-u)\big)\big(E_X^\delta(\bar{u}^\delta-u^\delta)\big)\label{eq:18}
\end{align}
\vspace*{-4ex}
\be \label{eq:19}
\begin{split}
\hspace{1em}+\big((E_Y^\delta)' \diff_t E_X^\delta(\bar{u}^\delta-u^\delta)&\big)\Big((A_Y^\delta)^{-1} (E_Y^\delta)' A E_X^\delta \bar{u}^\delta\\
&-(A_Y^\delta)^{-1}(E_Y^\delta)'(Au+\diff_t(u-E_X^\delta \bar{u}^\delta))\Big).
\end{split}
\ee
% I.p.v. het herschrijven van \gamma_T '\gamma_T E_X^\delta \bar{u}^\delta lijkt het aantrekkelijk om \gamma_0' \gamma_0 u te herschrijven, 
%want dan heb je Lemma~\ref{lem:inf-sup-identity} niet nodig.
% Echter je krijgt dan i.p.v. \big(\diff_t E_X^\delta (\bar{u}^\delta-u^\delta)\big) \big(E_X^\delta \bar{u}^\delta\big), een term 
%\big(\diff_t E_X^\delta (\bar{u}^\delta-u^\delta)\big) \big(u \big) en daar kan je niks mee.

Lipschitz continuity of $A \colon Y \rightarrow Y'$, and $\|\gamma_0 E_X^\delta\|_{X^\delta \rightarrow H}\leq 1$ and $\|\gamma_0\|_{X \rightarrow H}\leq 1$ by applications of  Corollary~\ref{corol:2}, show that the expression in \eqref{eq:18} is bounded by
\begin{align*}
L_A \|E_X^\delta \bar{u}^\delta-u\|_Y \|E_X^\delta(\bar{u}^\delta-u^\delta)\|_Y&+\|(E_Y^\delta)' \diff_t(E_X^\delta \bar{u}^\delta-u)\|_{(Y^\delta)'} \|\|E_X^\delta(\bar{u}^\delta-u^\delta)\|_Y\\
&+
\|E_X^\delta \bar{u}^\delta-u\|_{X}\|\bar{u}^\delta-u^\delta\|_{X^\delta}.
\end{align*}
By Lipschitz continuity of $(A_Y^\delta)^{-1}$ and $A$, the expression in \eqref{eq:19} is bounded by
\begin{align*}
\|(E_Y^\delta)' \diff_tE_X^\delta( \bar{u}^\delta-u^\delta)\|_{(Y^\delta)'} \tfrac{1}{m_A}\Big(L_A \|E_X^\delta \bar{u}^\delta-u\|_Y+\|(E_Y^\delta)' \diff_t(E_X^\delta \bar{u}^\delta-u)\|_{(Y^\delta)'}\big).
\end{align*}
Using these two upper bounds one infers that
\begin{align*}
    m_{S} \| \bar{u}^\delta&-u^\delta\|_{X^\delta}^2 \\
    & \leq \begin{pmatrix}
        L_A & 1 \\
        \frac{L_A}{m_A} & \frac{1}{m_A}
    \end{pmatrix} \begin{pmatrix}
        \|E_X^\delta \bar{u}^\delta-u\|_Y \\
        \|(E_Y^\delta)' \diff_t(E_X^\delta \bar{u}^\delta-u)\|_{(Y^\delta)'}
    \end{pmatrix} \cdot \begin{pmatrix}
        \|E_X^\delta(\bar{u}^\delta-u^\delta)\|_Y \\
        \|(E_Y^\delta)' \diff_tE_X^\delta( \bar{u}^\delta-u^\delta)\|_{(Y^\delta)'}
    \end{pmatrix}
    \\
    & \,\,\,\,\,\,+ \|E_X \bar{u}^\delta -u\|_X \|\bar{u}^\delta-u^\delta\|_{X^\delta}\\
    & \leq \big(1+ \sqrt{(1+L_A^2)(1 + \tfrac{1}{m_A^2})}\,\big) \| u-E_X^\delta \bar{u}^\delta\|_{X} \| \bar{u}^\delta-u^\delta\|_{X^\delta}  
\end{align*}
from which the proof follows easily.
\end{proof}

\begin{remark}In the case of a \emph{linear} $A\colon Y \rightarrow Y'$ with $A=A'$, a result as Theorem~\ref{thm:quasi-opt} can be found in \cite[Thm.~3.5]{249.99}, by using that for $X^\delta \subset Y^\delta$ and such $A$ it holds that $u^\delta=\argmin_{z \in X^\delta}\|(\ell,u_0)-B_e E_X^\delta z\|_{(Y^\delta)' \times H}^2$ (see \cite[Rem.~4.1]{249.992}). Apart from being more generally applicable, the current derivation is more transparant. 

For a plain Galerkin discretization of the heat equation with homogeneous initial condition, an estimate as in Theorem~\ref{thm:quasi-opt} (with `$Y^\delta$'$=X^\delta$ in the definition of $\|\bigcdot\|_{X,\delta}$) can already be found in \cite[Thm.~3.2]{249.2}.
\end{remark}

\begin{corollary} \label{corol:1} For $X^\delta \subset Y^\delta$, it holds that
$$
\|\theta^\delta-u^\delta\|_{Y^\delta} +\tfrac{\sqrt{1+L_A^2}}{m_A} \|u_0-\gamma_0  E_X^\delta u^\delta\|_H \leq 2 C_1 \tfrac{\sqrt{1+L_A^2}}{m_A} \inf_{\bar{u}^\delta \in X^\delta} \|u-E_X^\delta \bar{u}^\delta\|_X.
$$
\end{corollary}

\begin{proof} It holds that
\begin{align*}
m_A \|\theta^\delta-u^\delta\|_{Y^\delta} & \leq \|A_Y^\delta \theta^\delta- A_Y^\delta u^\delta\|_{(Y^\delta)'} =\|(E_Y^\delta)'(\ell-B E_X^\delta u^\delta)\|_{(Y^\delta)'}
\\&=\|(E_Y^\delta)'(\diff_t(u-E_X^\delta u^\delta)+A u -A E_X^\delta u^\delta)\|_{(Y^\delta)'}
\\& \leq
\|(E_Y^\delta)'(\diff_t(u-E_X^\delta u^\delta)\|_{(Y^\delta)'}+L_A \|u-E_X^\delta u^\delta\|_{Y} \\
& \leq \sqrt{1+L_A^2} \|u-E_X^\delta u^\delta\|_{X,\delta}.
\end{align*}
The proof is completed by an application of Theorem~\ref{thm:quasi-opt}.
\end{proof}

In the framework of minimal residual approximations w.r.t.~a discretized dual norm for the solution of a linear PDE, the principle behind the following theorem was introduced in \cite[Lemma~3.3]{45.44}, and in a modified form, further investigated in \cite{64.18}.

\begin{theorem}[Quasi-optimal Galerkin approximation] \label{thm:pjotr}
Let $X^\delta \subset Y^\delta$. For $\varrho>0$ being any constant, suppose that, 
with
$$
\theta_\delta:=A^{-1}(\ell - \diff_t E_X^\delta u^\delta),
$$
it holds that
\be \label{eq:pjotr}
\|\theta_\delta- E_Y^\delta \theta^\delta\|_Y \leq \varrho \big( \|\theta^\delta- u^\delta\|_{Y^\delta} + \tfrac{\sqrt{1+L_A^2}}{m_A} \|u_0-\gamma_0  E_X^\delta u^\delta\|_H\big).
\ee
Then
$$
\|u-E_X^\delta u^\delta\|_X \leq \big(\tfrac{L_S}{m_S} +L_{\NN^{-1}} L_A \varrho \tfrac{2 C_1 \sqrt{1+L_A^2}}{m_A}\big) \inf_{\bar{u}^\delta \in X^\delta}\|u- E_X^\delta \bar{u}^\delta\|_X.
$$
\end{theorem}

\begin{proof} We consider the operator $\NN_\delta := \left[\begin{array}{@{}cc@{}} A & \diff_t E_X^\delta\\ (E_X^\delta)' \diff_t'& -(A_X^\delta+(\gamma^\delta_T)' \gamma^\delta_T)\end{array}\right]\colon Y \times X^\delta \rightarrow Y' \times (X^\delta)'$. 
Let $(\underline{\theta}_\delta,\underline{u}^\delta) \in Y \times X^\delta$ solve
\be \label{eq:38}
\NN_\delta \left[\begin{array}{@{}c@{}} \underline{\theta}_\delta \\ \underline{u}^\delta \end{array}\right]=\left[\begin{array}{@{}c@{}} \ell \\  (E^\delta_X)' (-\ell-\gamma_0' u_0) \end{array}\right].
\ee
By eliminating $\underline{\theta}_\delta$, one infers that $E_X^\delta S (E_X^\delta)' \underline{u}^\delta=0$ with $S$ from \eqref{eq:37},\footnote{Reading $\delta=\delta_*$. Notice that $E_X^\delta S (E_X^\delta)' \neq S^\delta$ (unless $\delta=\delta_*$).} i.e., $\underline{u}^\delta$ is the Galerkin approximation of the solution $u$ of $Su=0$.

Since $S$ is Lipschitz continuous and strongly monotone, it is well-known  that $E_X^\delta \underline{u}^\delta$ is a quasi-best approximation for $u$ from $X^\delta$. Indeed, for arbitrary $\bar{u}^\delta \in X^\delta$, we have
\begin{align*}
 m_S \|u-E_X^\delta \underline{u}^\delta\|_X^2& \leq (S u-S E_X^\delta \underline{u}^\delta)(u-E_X^\delta \underline{u}^\delta)\\
& =(S u-S E_X^\delta \underline{u}^\delta)(u-E_X^\delta \bar{u}^\delta)
\leq L_S \|u-E_X^\delta \underline{u}^\delta\|_X\|u-E_X^\delta \bar{u}^\delta\|_X,
\end{align*}
or
\be \label{eq:16}
\|u-E_X^\delta \underline{u}^\delta\|_X \leq \tfrac{L_S}{m_S}\inf_{\bar{u}^\delta \in X^\delta}\|u-E_X^\delta \bar{u}^\delta\|_X.
\ee

From the second equation in $\left[\begin{array}{@{}cc@{}} A_Y^\delta & \diff_t^\delta\\ (\diff_t^\delta)'& -(A_X^\delta+(\gamma^\delta_T)' \gamma^\delta_T)\end{array}\right] \left[\begin{array}{@{}c@{}} \theta^\delta \\ u^\delta \end{array}\right]=\left[\begin{array}{@{}c@{}} (E^\delta_Y)' \ell \\  (E^\delta_X)' (-\ell-\gamma_0' u_0) \end{array}\right]$, and the definition of $\theta_\delta$, we have
\be \label{eq:24}
\begin{split}
&\NN_\delta \left[\begin{array}{@{}c@{}} E_Y^\delta \theta^\delta \\ u^\delta \end{array}\right] -\NN_\delta \left[\begin{array}{@{}c@{}} \underline{\theta}_\delta \\ \underline{u}^\delta \end{array}\right]\\
&=\left[\begin{array}{@{}c@{}} A E_Y^\delta \theta^\delta +\diff_t E_X^\delta u^\delta \\ (E_X^\delta)' \diff_t' E_Y^\delta \theta^\delta-(A_X^\delta+(\gamma^\delta_T)' \gamma^\delta_T) u^\delta\end{array}\right]
-
\left[\begin{array}{@{}c@{}} \ell \\ (E_X^\delta)'(-\ell-\gamma_0' u_0) \end{array}\right]\\
&=
\left[\begin{array}{@{}c@{}} A E_Y^\delta \theta^\delta -A \theta_\delta\\ 0
\end{array}\right].
\end{split}
\ee
Because Corollary~\ref{corol:stability} also applies to $\NN_\delta$, one infers that
\be \label{eq:17}
\begin{split}
\|E_X^\delta(u^\delta-\underline{u}^\delta)\|_{X} & \leq \|E_Y^\delta \theta^\delta-\underline{\theta}_\delta\|_Y+\|E_X^\delta(u^\delta-\underline{u}^\delta)\|_{X} 
\\ &\leq 
L_{\NN^{-1}} \left\| \NN_\delta \left[\begin{array}{@{}c@{}} E_Y^\delta \theta^\delta \\ u^\delta \end{array}\right] -\NN_\delta \left[\begin{array}{@{}c@{}} \underline{\theta}_\delta \\ \underline{u}^\delta \end{array}\right] \right\|_{Y' \times (X^\delta)'} \\
& \leq L_{\NN^{-1}} L_A \|E_Y^\delta \theta^\delta -\theta_\delta\|_Y \\
&\leq L_{\NN^{-1}} L_A \varrho \big( \|\theta^\delta- u^\delta\|_{Y^\delta} + \tfrac{\sqrt{1+L_A^2}}{m_A} \|u_0-\gamma_0  E_X^\delta u^\delta\|_H\big)\\
& \leq L_{\NN^{-1}} L_A \varrho \tfrac{2 C_1 \sqrt{1+L_A^2}}{m_A} \inf_{\bar{u}^\delta \in X^\delta} \|u-E_X^\delta \bar{u}^\delta \|_X
\end{split}
\ee
by Corollary~\ref{corol:1}.
The proof is completed by the combination of \eqref{eq:16} and \eqref{eq:17} and a triangle inequality.
\end{proof}

Under the harmless assumption that $A(0)=0$,
next we show that, for $u \not\in X^\delta$, \eqref{eq:pjotr} will always be true by taking $Y^\delta \supset X^\delta$ sufficiently large.  Although $\theta^\delta$ is the Galerkin approximation from $Y^\delta$ to the solution of $A \theta_\delta=\ell-\diff_t E_X^\delta u^\delta$, this statement is not immediate because also $u^\delta$ at the right-hand side in \eqref{eq:pjotr} depends on $Y^\delta$.

\begin{proposition} \label{prop:2} Let $A(0)=0$, and $(\delta_i)_{i \in \N}$ such that $Y^{\delta_i} \subset Y^{\delta_{i+1}}$ and $\overline{\cup_{i \in \N} Y^{\delta_i}}=Y$. Then for $(\ell, u_0) \in Y' \times H$ and a (fixed) finite dimensional $X^\delta$, for $u \not\in X^\delta$ condition \eqref{eq:pjotr} is valid for $Y^\delta  \supset Y^{\delta_i} \cup X^\delta$ and $i$ being sufficiently large.
\end{proposition}

\begin{proof}
Since $A$ is Lipschitz continuous and strongly monotone, it is well-known that $\theta^\delta$ is a quasi-best approximation for $\theta_\delta$ from $Y^\delta$, cf.~\eqref{eq:16}.
As a consequence of $N^\delta 0=0$ and $\|(\gamma_0 E_X^\delta)'\|_{H \rightarrow (X^\delta)'} \leq 1$ by Corollary~\ref{corol:2}, Corollary~\ref{corol:stability} shows that $\|u^\delta\|_{X^\delta} \leq L_{N^{-1}}(2 \|\ell\|_{Y'}+\|u_0\|_H)$, i.e., $u^\delta$ is bounded uniformly in $Y^\delta$.
So since $X^\delta$ is a finite dimensional space, the set of these $u^\delta$ is precompact in $X$, and so because $\diff_t\colon X \rightarrow Y'$ and $A^{-1}\colon Y' \rightarrow Y$ are continuous, the corresponding set of solutions $\theta_\delta$ is precompact in $Y$. We conclude that for $Y^\delta  \supset Y^{\delta_i}$ and $i \rightarrow \infty$, it holds that $\|\theta_\delta-E_Y^\delta \theta^\delta\|_Y \rightarrow 0$.

Considering the right-hand side of \eqref{eq:pjotr}, with $\underline{u}^\delta \in X^\delta$ defined in \eqref{eq:38},
from \eqref{eq:24} it follows that $\|E_Y^\delta \theta^\delta - \theta_\delta\|_Y \rightarrow 0$ implies that $\|u^\delta  - \underline{u}^\delta\|_{X^\delta} \rightarrow 0$, so that $\|E_X^\delta(u^\delta  - \underline{u}^\delta)\|_{Y} \rightarrow 0$.
 Corollary~\ref{corol:Lipstab} shows that
\begin{align*}
L_A \|\theta_\delta - E_X^\delta \underline{u}^\delta\|_{Y}+\|u_0 -\gamma^\delta_0 E_X^\delta \underline{u}^\delta\|_H &\geq \|A \theta_\delta - A E_X^\delta \underline{u}^\delta\|_{Y'}+\|u_0 -\gamma^\delta_0 E_X^\delta \underline{u}^\delta\|_H\\
& =  \|\ell - (\diff_t+A) E_X^\delta \underline{u}^\delta\|_{Y'}+\|u_0 -\gamma^\delta_0 E_X^\delta \underline{u}^\delta\|_H\\
&\geq L_{B_e^{-1}}^{-1} \|u-E_X^\delta \underline{u}^\delta\|_X,
\end{align*}
which shows that, in contrast to the left-hand side, for $u \not\in X^\delta$ the right-hand side of \eqref{eq:pjotr} does \emph{not} tend to zero when $Y^\delta  \supset Y^{\delta_i}$ and $i \rightarrow \infty$, and thus completes the proof.
\end{proof}

\begin{remark} In the unlikely event that $u \in X^\delta$, \eqref{eq:pjotr} might never be satisfied. By taking $Y^\delta  \supset Y^{\delta_i}$ and letting $i \rightarrow \infty$, however, $u^\delta$ converges to $ \underline{u}^\delta$, which in this case equals $u$ by \eqref{eq:16}.
\end{remark}

The combination of Theorem~\ref{thm:pjotr} and Proposition~\ref{prop:2} suggests an iterative enlargement of $Y^\delta$ until \eqref{eq:pjotr} is satisfied.
The right-hand side of \eqref{eq:pjotr} is available, but since $\theta_\delta$ is not, in order to do so 
one needs a computable \emph{a posteriori error estimator} for the Galerkin error $\|\theta_\delta-\theta^\delta\|_Y$.
Preferably it is built from local error indicators, which allows an efficient \emph{adaptive refinement}.
Although the operator $A$ involved has the nice properties of being Lipschitz continuous and strongly monotone, because of the anisotropic nature of the space $Y=L_2(\Xi;V)$ the development of such an estimator is non-trivial. We postpone its construction for the quasi-linear parabolic PDEs from Example~\ref{ex:quasi} to a forthcoming work.
\medskip

A complementary advantage of guaranteeing quasi-optimality by enforcing \eqref{eq:pjotr} is 
that it permits to compute an a \emph{posteriori error estimator} for $\|u-E_X^\delta u^\delta\|_X$, which can be employed for an \emph{adaptive refinement} of $X^\delta$ in an \emph{outer iteration}.

\begin{proposition}[Efficiency and reliability] \label{prop:5} Let $X^\delta \subset Y^\delta$, and let condition~\eqref{eq:pjotr} be satisfied. Then, for $u \not\in X^\delta$, it holds that
\begin{align*}
\tfrac{m_A}{\sqrt{1+L_A^2 +m_A^2}} \leq &\frac{\|u-E_X^\delta u^\delta\|_X}{\sqrt{\|\theta^\delta-u^\delta\|_{Y^\delta}^2+\|u_0-\gamma^\delta_0 u^\delta\|_H^2}}\\
&\hspace*{5em} \leq L_{B_e^{-1}}
\sqrt{L_A^2(\varrho+1)^2+(L_A \varrho \tfrac{\sqrt{1+L_A^2}}{m_A}+1)^2}.
\end{align*}
\end{proposition}

\begin{proof} Thanks to $X^\delta \subset Y^\delta$,
\begin{align*}
\|\theta^\delta- u^\delta\|_{Y^\delta}&=\|(A_Y^\delta)^{-1} ((E_Y^\delta)'\ell -\diff_t^\delta u^\delta)-(A_Y^\delta)^{-1} A_Y^\delta u^\delta\|_{Y^\delta}\\
&\leq \tfrac{1}{m_A} \|(E_Y^\delta)'\big(\ell -(\diff_t + A) E_X^\delta u^\delta\big)\|_{(Y^\delta)'}\\
&\leq \tfrac{1}{m_A} \|(\diff_t + A)u -(\diff_t + A) E_X^\delta u^\delta\|_{Y'}\\
& \leq \tfrac{1}{m_A}\big(\|\diff_t(u- E_X^\delta u^\delta)\|_{Y'}+ L_A\|u-E_X^\delta u^\delta\|_Y\big) \leq \sqrt{\tfrac{1+L_A^2}{m_A^2}}\, \|u-E_X^\delta  u^\delta\|_X,
\end{align*}
which, in combination with $\|\gamma_0\|_{X \rightarrow H} \leq 1$ (Corollary~\ref{corol:2}), shows the lower bound.

Using Corollary~\ref{corol:Lipstab}, the upper bound follows from
\begin{align*}
\|u-E_X^\delta u^\delta\|_X&=\|B_e^{-1} B_e u-B_e^{-1} B_e E_X^\delta u^\delta\|_X \\
&\leq L_{B_e^{-1}} \big(\|\ell - (\diff_t+A) E_X^\delta u^\delta\|_{Y'}+\|u_0 -\gamma^\delta_0 u^\delta\|_H\big)
\\ & = L_{B_e^{-1}} \big(\|A \theta_\delta - A E_X^\delta u^\delta\|_{Y'}+\|u_0 -\gamma^\delta_0 u^\delta\|_H\big)
\\ & \leq L_{B_e^{-1}} \big(L_A \|\theta_\delta - E_X^\delta u^\delta\|_{Y}+\|u_0 -\gamma^\delta_0 u^\delta\|_H\big)
\\ & \leq L_{B_e^{-1}} \big(L_A (\|\theta_\delta - \theta^\delta\|_{Y}+\|\theta^\delta - u^\delta\|_{Y^\delta})+\|u_0 -\gamma^\delta_0 u^\delta\|_H\big)
\\ & \leq L_{B_e^{-1}} \big(L_A (\varrho+1) \|\theta^\delta-u^\delta\|_{Y^\delta}+(L_A \varrho \tfrac{\sqrt{1+L_A^2}}{m_A}+1) \|u_0 -\gamma^\delta_0 u^\delta\|_H\big).\qedhere
\end{align*}
\end{proof}

%%%%%%%%%%%%%%
\section{Inexact Uzawa} \label{sec:7}
%%%%%%%%%%%%%%
We have seen that $u:=B_e^{-1}(\ell,u_0)$ solves $\NN
\left[\begin{array}{@{}c@{}} \theta \\ u \end{array}\right]=
\left[\begin{array}{@{}c@{}} \ell \\ -(\ell+\gamma_0' u_0) \end{array}\right]$, where $\theta=u$, and we have studied the accuracy of Galerkin discretizations of this system. In the current section we present and analyze an iterative method for solving such nonlinear Galerkin systems, i.e.,
$$
\NN^\delta \left[\begin{array}{@{}c@{}} \theta^\delta \\ u^\delta \end{array}\right]=\left[\begin{array}{@{}c@{}} f^\delta \\ g^\delta \end{array}\right],
$$
where $(f^\delta,g^\delta) \in (Y^\delta)' \times (X^\delta)'$.
The method we are considering is a generalization of the well-known inexact Uzawa iteration for linear (generalized) saddle point problems (e.g.~see \cite{34.67,322,234.7}) to the nonlinear saddle point problem \eqref{eq:Galsystem}. The adjective `inexact' refers to the fact that in these iterations the exact inverse of the `left upper block' (in our case the nonlinear operator $A_Y^\delta$)
is avoided.

From Proposition~\ref{prop:equivSchur} recall that \eqref{eq:Galsystem} is equivalent to 
$$
S^\delta u^\delta=0,\quad A_Y^\delta \theta^\delta =f^\delta-\diff_t^\delta u^\delta,
$$
where for $z \in X^\delta$, $S^\delta z  :=A_X^\delta z+(\gamma_T^\delta)'\gamma_T^\delta z+g^\delta-(\diff_t^\delta)' (A_Y^\delta)^{-1}(f^\delta-\diff_t^\delta z)$. Further, recall that $A_Y^\delta \colon Y^\delta \rightarrow (Y^\delta)'$ is Lipschitz continuous and strongly monotone with constants $L_A$ and $m_A$, and as shown in Lemma~\ref{lem:Schur},
$S^\delta\colon X^\delta \rightarrow (X^\delta)'$ is Lipschitz continuous and strongly monotone with constants $L_S$ and $m_S$.

Recalling our discussion of the Zaratonello iteration in Sect.~\ref{sec:LipMono}, these properties of  $A_Y^\delta$ and $S^\delta$ imply that with
$$
\zeta^*_S:=\frac{m_S}{L_S^2},\quad \sigma_S:=\sqrt{1-\frac{m_S^2}{L_S^2}},\quad \zeta^*_A:=\frac{m_A}{L_A^2},\quad \sigma_A:=\sqrt{1-\frac{m_A^2}{L_A^2}},
$$
and the Riesz maps $R_{X^\delta}\colon X^\delta \rightarrow (X^\delta)'$  and $R_{Y^\delta}\colon Y^\delta \rightarrow (Y^\delta)'$, for $z \in X^\delta$ it holds that
\begin{align*}
\| u^\delta-\big(z-\zeta^*_S R_{X^\delta}^{-1} S^\delta z\big)\|_{X^\delta} & \leq \sigma_S \|u^\delta-z\|_{X^\delta},
\intertext{and with $\mu^\delta:=(A_Y^\delta)^{-1}(f^\delta -\diff_t^\delta z)$ for some $z \in X^\delta$, for $\xi \in Y^\delta$ it holds that}
\| \mu^\delta-\big(\xi-\zeta^*_A R_{Y^\delta}^{-1} \big[A_Y^\delta \xi-(f^\delta - \diff_t^\delta z)\big]\big)\|_{Y^\delta} &\leq \sigma_A \|\mu^\delta - \xi\|_{Y^\delta}.
\end{align*}

A problem is that given an approximation $z$ for $u^\delta$, the improved approximation provided by the exact Zaratonello iterand $z-\zeta^*_S R_{X^\delta}^{-1} S^\delta z$ is not computable because of the presence of the term $(A_Y^\delta)^{-1}(f^\delta-\diff_t^\delta z)$ in the expression for $S^\delta z$. The idea of the inexact Uzawa iteration is to approximate this term, with a sufficient accuracy, by the application of a few iterations of some iterative method, in our case Zaratonello, for approximating $(A_Y^\delta)^{-1}(f^\delta-\diff_t^\delta z)$. The starting value of these iterations  is the analogously previously computed approximation of $(A_Y^\delta)^{-1}(f^\delta-\diff_t^\delta \widehat{z})$, with 
$\widehat{z}$ being the previous approximation for $u^\delta$.

The resulting algorithm reads as follows.

\begin{algorithm}[Inexact Uzawa] \label{alg:InexactUzawa}. 
{\rm%
\begin{algotab}
\> Let $L \in \N$ be a constant that will be specified in Theorem~\ref{thm:Uzawa}.\\
\> Let $(\theta^{(0)},u^{(0)}) \in Y^\delta \times X^\delta$ be given. \\
\> \texttt{for} \= $k=0,1,\ldots$ \texttt{do}\\
\>\> $\theta_0^{(k+1)}:=\theta^{(k)}$\\
\>\> \texttt{for} \= $i=0,1,\ldots L-1$ \texttt{do}\\
\>\>\> $\theta_{i+1}^{(k+1)}:=\theta_{i}^{(k+1)}- \zeta^*_A R_{Y^\delta}^{-1} \big[A_Y^\delta \theta_{i}^{(k+1)}-(f^\delta - \diff_t^\delta u^{(k)})\big]\big)$\\
\>\> \texttt{endfor} \\
\>\>  $\theta^{(k+1)}:=\theta_L^{(k+1)}$\\
\>\>  $u^{(k+1)}:=u^{(k)}-\zeta_S^* R_{X^\delta}^{-1}\big[A_X^\delta u^{(k)} +(\gamma_T^\delta)'\gamma_T^\delta u^{(k)}+g^\delta-(\diff_t^\delta)'\theta^{(k+1)}\big]$\\
\> \texttt{endfor}
\end{algotab}}
\end{algorithm}

\begin{theorem}[Convergence and a priori error estimate] \label{thm:Uzawa}
For $\widehat{\sigma}_S \in (\sigma_S,1)$, let $C_3:=\frac{1}{\widehat{\sigma}_S}\big(\frac{\widehat{\sigma}_S-\sigma_S}{\zeta_S^*}+\frac{1}{m_A}\big)$, and let $L \in \N$ be such that 
$\sigma_A^L(C_3+\frac{1}{m_A}) \leq \frac{\widehat{\sigma}_S-\sigma_S}{\zeta_S^*}$. Then with $C_4:=\max\{\frac{1}{C_3}\|\theta^\delta-\theta^{(0)}\|_{Y^\delta},\|u^\delta-u^{(0)}\|_{X^\delta}\}$, it holds that
\be \label{eq:uzawa}
\tfrac{1}{C_3} \|\theta^\delta-\theta^{(k)}\|_{Y^\delta} \leq \widehat{\sigma}_S^k C_4, \quad
\|u^\delta-u^{(k)}\|_{X^\delta} \leq \widehat{\sigma}_S^k C_4.
\ee
\end{theorem}

\begin{proof} The statement is valid for $k=0$, and let it be valid for some $k \in \N$.
Then using $\|\diff_t^\delta\|_{X^\delta \rightarrow (Y^\delta)'} \leq 1$, we have
\begin{align*}
\|(A_Y^{\delta})^{-1}(f^\delta-&\diff_t^\delta u^{(k)})-\theta^{(k)}\|_{Y^\delta}\\ &\leq
\|\theta^\delta-\theta^{(k)}\|_{Y^\delta}+\|(A_Y^{\delta})^{-1}(f^\delta-\diff_t^\delta u^\delta)-(A_Y^{\delta})^{-1}(f^\delta-\diff_t^\delta u^{(k)})\|_{Y^\delta}\\
&\leq C_3 \widehat{\sigma}_S^k C_4+ \tfrac{1}{m_A}  \widehat{\sigma}_S^k C_4,
\end{align*}
and so
$$
\|(A_Y^{\delta})^{-1}(f^\delta-\diff_t^\delta u^{(k)})-\theta^{(k+1)}\|_{Y^\delta} \leq \sigma_A^L \big(C_3+\tfrac{1}{m_A}\big) \widehat{\sigma}_S^k C_4.
$$
We infer that
\begin{align*}
\|\theta^\delta - \theta^{(k+1)}\|_{Y^\delta}  \leq& \|(A_Y^{\delta})^{-1}(f^\delta-\diff_t^\delta u^{(k)} )-\theta^{(k+1)}\|_{Y^\delta}\\
&+\|(A_Y^{\delta})^{-1}(f^\delta-\diff_t^\delta u^{\delta})-(A_Y^{\delta})^{-1}(f^\delta-\diff_t^\delta u^{(k)})\|_{Y^\delta}\\
 \leq &(\sigma_A^L(C_3 + \tfrac{1}{m_A})+\tfrac{1}{m_A})\widehat{\sigma}_S^k C_4 \leq C_3 \widehat{\sigma}_S^{k+1} C_4,
\end{align*}
and
\begin{align*}
\|u^\delta-u^{(k+1)}\|_{X^\delta} \leq & \|u^\delta-\big(u^{(k)}-\zeta_S^* R_{X^\delta}^{-1} S^\delta u^{(k)}\big)\|_{X^\delta}
\\
&+\zeta_S^* \|(\diff_t^\delta)'\big[(A_Y^\delta)^{-1}(f^\delta-\diff_t^\delta u^{(k)})-\theta^{(k+1)}\big]\|_{(X^\delta)'}\\
 \leq & \sigma_S \|u^\delta-u^{(k)}\|_{X^\delta}+ \zeta_S^* \sigma_A^L(C_3 + \tfrac{1}{m_A}) \widehat{\sigma}_S^k C_4\\
 \leq & \big(\sigma_S+\zeta_S^* \sigma_A^L(C_3 + \tfrac{1}{m_A})\big) \widehat{\sigma}_S^k C_4\\
\leq &\widehat{\sigma}_S^{k+1} C_4. \qedhere
\end{align*}
\end{proof}

\begin{remark} Inexact Uzawa type algorithms for classes of nonlinear systems can be found in the literature. We are, however, not aware of an analysis similar to Theorem~\ref{thm:Uzawa} of its convergence exclusively in terms of the Lipschitz and monotonicity constants of $A_Y^\delta$ and `generalized Schur complement' $S^\delta$, analogously to the analysis in \cite{34.67} for the linear case.
\end{remark}

\begin{remark} For quasi-linear parabolic PDEs it might be beneficial to replace the inner Zarantonello iteration by a faster converging iteration like the Ka\v{c}anov iteration, or the (damped) Newton iteration (see \cite{137,137.6}). With those iterations, however, the linear operator that has to be inverted changes per iteration, which, in any case with the application of direct linear solvers, strongly inflates the cost. Due to the structure of the operator $S^\delta$, it is less obvious to find an alternative for the outer Zarantonello iteration.
\end{remark}

\subsection{A posteriori error estimation}
The a priori upper bounds \eqref{eq:uzawa} for the error in $(\theta^{(k)},u^{(k)})$ can be expected to be increasingly pessimistic for $k \rightarrow \infty$.
We therefore provide an a posteriori computable error bound.

\begin{proposition}[A posteriori error estimate] \label{prop:3} Let $(\tilde{\theta}^\delta,\tilde{u}^\delta) \in Y^\delta \times X^\delta$ approximate the solution $(\theta^\delta,u^\delta)$ of \eqref{eq:Galsystem}. Then with
$$
\left[\begin{array}{@{}c@{}} r_Y^\delta \\ r_X^\delta \end{array}\right]
=\left[\begin{array}{@{}c@{}} r_Y^\delta(\tilde{\theta}^\delta,\tilde{u}^\delta)  \\ r_X^\delta(\tilde{\theta}^\delta,\tilde{u}^\delta)  \end{array}\right]
:=\left[\begin{array}{@{}c@{}} f^\delta \\ g^\delta \end{array}\right] -\NN^\delta \left[\begin{array}{@{}c@{}} \tilde{\theta}^\delta \\ \tilde{u}^\delta \end{array}\right],
$$
for $(\tilde{\theta}^\delta,\tilde{u}^\delta) \neq (\theta^\delta,u^\delta)$ it holds that
$$
L_N^{-1} \leq
 \frac{\|\theta^\delta-\tilde{\theta}^\delta\|_{Y^\delta}+\|u^\delta-\tilde{u}^\delta\|_{X^\delta}}
{\sqrt{r_Y^\delta(R_{Y^\delta}^{-1}r_Y^\delta)}+\sqrt{r_X^\delta(R_{X^\delta}^{-1}r_X^\delta)}} \leq L_{N^{-1}}.
$$
\end{proposition}

\begin{proof} Applications of Proposition~\ref{prop:Lip} and Corollary~\ref{corol:stability} show that
$$
L_N^{-1} \leq
 \frac{\|\theta^\delta-\tilde{\theta}^\delta\|_{Y^\delta}+\|u^\delta-\tilde{u}^\delta\|_{X^\delta}}
{\|r_Y^\delta\|_{(Y^\delta)' }+\|r_X^\delta\|_{(X^\delta)'}}\leq L_{N^{-1}},
$$
whereas for $Z\in \{Y,X\}$, $\|r_Z^\delta\|_{(Z^\delta)'}^2=\|(R_Z^{\delta})^{-1} r_Z^\delta\|_{Z^\delta}^2=r_Z^\delta((R_Z^{\delta})^{-1} r_Z^\delta)$.
\end{proof}

When we apply this a posteriori error estimate to $(\tilde{\theta}^\delta,\tilde{u}^\delta)=(\theta^{(k+1)},u^{(k)})$ produced in the inexact Uzawa iteration, it requires computing
$$
r_Y^\delta:=r_Y^\delta(\theta^{(k+1)},u^{(k)}),\quad R_{Y^\delta}^{-1}(r_Y^\delta),\quad  r_X^\delta:=r_X^\delta(\theta^{(k+1)},u^{(k)}), \quad R_{X^\delta}^{-1}(r_X^\delta),
$$
as well as  $r_Y^\delta\big(R_{Y^\delta}^{-1}(r_Y^\delta)\big)$ and $r_X^\delta\big(R_{X^\delta}^{-1}(r_X^\delta)\big)$. The computation of the latter two boils down to computing 
the inner product between the vectors that represent the residual and the `preconditioned' residual.

If the a posteriori error estimate shows an insufficient accuracy of $(\theta^{(k+1)},u^{(k)})$, then $u^{(k+1)}$ is obtained as $u^{(k)}-\zeta^* R_{X^\delta}^{-1}(r_X^\delta)$, with both terms already being available. For the computation of $\theta_1^{(k+1)}$  in the next inner iteration one saves the computation of $A_Y^\delta \theta^{(k+1)}-f^\delta$, because it was already computed for $r_Y^\delta$. In other words, monitoring the error in $(\theta^{(k+1)},u^{(k)})$ while performing the inexact Uzawa iteration adds a small percentage to its cost, whereas it allows to stop as soon as the desired accuracy is attained.
\medskip

\subsection{The application of the inverse Riesz maps}
The execution of the inexact Uzawa algorithm as well as the computation of the a posteriori error estimate requires recurrent applications of $R_{Y^\delta}^{-1}$ and $R_{X^\delta}^{-1}$, which we discuss now.
It holds that $R_{Y^\delta}=(E_Y^\delta)' R_Y E_Y^\delta$. So for $h^\delta \in (Y^\delta)'$, $v^\delta=R_{Y^\delta}^{-1} h^\delta$ is the solution of the discretized (linear) coercive problem
$$
\langle E_Y^\delta v^\delta,E_Y^\delta z^\delta\rangle_Y=h^\delta(z^\delta)\quad  (z^\delta \in Y^\delta). 
$$

The definition of $\|\bigcdot\|_{X^\delta}$ shows that
$$
R_{X^\delta}=(E_X^\delta)' R_Y E_X^\delta+(\gamma_T^\delta)' \gamma_T^\delta+(\diff_t^\delta)'((E_Y^\delta)' R_Y E_Y^\delta)^{-1} \diff_t^\delta,
$$
so that
$$
(R_{X^\delta})^{-1}=-
\left[\begin{array}{@{}cc@{}} 0 & \identity \end{array} \right]
\left[\begin{array}{@{}cc@{}} (E_Y^\delta)' R_Y E_Y^\delta & \diff_t^\delta\\ (\diff_t^\delta)'& -((E_X^\delta)' R_Y E_X^\delta+(\gamma^\delta_T)' \gamma^\delta_T)\end{array}\right]^{-1}
\left[\begin{array}{@{}c@{}} 0 \\ \identity \end{array}  \right].
$$
So for $h^\delta \in (X^\delta)'$, $v^\delta=R_{X^\delta}^{-1} h^\delta$ is found as the second component of the solution
of a discretized (linear) parabolic problem in the form
$$
\left[\begin{array}{@{}cc@{}} (E_Y^\delta)' R_Y E_Y^\delta & \diff_t^\delta\\ (\diff_t^\delta)'& -((E_X^\delta)' R_Y E_X^\delta+(\gamma^\delta_T)' \gamma^\delta_T)\end{array}\right]
\left[\begin{array}{@{}c@{}}  \varsigma^\delta \\ v^\delta \end{array}\right]=\left[\begin{array}{@{}c@{}}  0 \\ -h^\delta \end{array}\right].
$$

\subsection{The application of optimal preconditioners}
For $Z$ reading as either $Y$ or $X$, let $\tilde{R}_{Z^\delta}\colon Z^\delta \rightarrow (Z^\delta)'$  be such that $(\tilde{R}_{Z^\delta} \bigcdot)(\bigcdot)^\frac12$  is a scalar product on $Z^\delta \times Z^\delta$ whose induced norm is equivalent, uniformly in $\delta \in \Delta$, to $\|\bigcdot\|_{Z^\delta}$, and such that $\tilde{R}_{Z^\delta}^{-1}$, known as a \emph{preconditioner}, can be applied `efficiently', in particular preferably at linear cost.
Then the application of the exact inverses of $R_{Y^\delta}$ and $R_{X^\delta}$ in the inexact Uzawa algorithm can be \emph{replaced} by the application of these preconditioners $\tilde{R}_{Y^\delta}$ and $\tilde{R}_{X^\delta}$.
After making obvious adaptations to the constants $m_A$, $L_A$ and $m_S$, $L_S$ that account for the replacement of the norms $\|\bigcdot\|_{Y^\delta}$ and $\|\bigcdot\|_{X^\delta}$ by $(\tilde{R}_{Y^\delta} \bigcdot)(\bigcdot)^\frac12$ and $(\tilde{R}_{X^\delta} \bigcdot)(\bigcdot)^\frac12$, respectively, the analysis of the resulting algorithm applies without modification.

\begin{remark} The above is an alternative for the \emph{approximate} application of the inverse Riesz operator by means of an \emph{inner iteration} that involves the repeated application of a preconditioner. The latter has been analyzed in the literature in the framework of numerically solving quasi-linear elliptic PDEs in, e.g.~\cite{81}.
\end{remark}

The issue of the construction of optimal preconditioners $(\tilde{R}_{Z^\delta}^{-1})_{\delta \in \Delta}$ depends on $(Z^\delta)_{\delta \in \Delta} \subset Z$, but not on the operators $A^\delta_Y$ or $S^\delta$, and so it is the same issue as studied earlier in the literature in the linear case.

For the tensor product setting from Example~\ref{ex:tensor}, in Sect.~\ref{sec:precond} we recall a construction of optimal preconditioners  based on ideas from \cite{12.5} (see also \cite{234.7}, \cite[Sect.~5.6]{249.991}).

%%%%%%%%%%%%%%
\section{Conclusion and outlook} \label{sec:8}
%%%%%%%%%%%%%%
For $V \hookrightarrow H \simeq H' \hookrightarrow V'$, $Y:=L_2(\Xi;V)$, $X:=Y \cap H^1(\Xi;V')$, $A \colon Y \rightarrow Y'$ Lipschitz continuous and strongly monotone, and data $(\ell,u_0) \in Y' \times H$, 
$$
\left[\begin{array}{@{}cc@{}} A & \diff_t\\ \diff_t'& -(A+\gamma_T' \gamma_T)\end{array}\right]
\left[\begin{array}{@{}c@{}} \theta \\ u \end{array}\right]=
\left[\begin{array}{@{}c@{}} \ell \\ -(\ell+\gamma_0' u_0) \end{array}\right].
$$
is an equivalent formulation of the IVP in time-space variational form $\diff_t u+Au =\ell $, $\gamma_0 u=u_0$, where $\theta=u$.

For families $(Y^\delta,X^\delta)_{\delta \in \Delta}$ of closed subspaces $Y^\delta \supset X^\delta$ of $Y$ and $X$ that satisfy a uniform inf-sup condition, the $u^\delta$-component of the Galerkin approximation $(\theta^\delta,u^\delta)$ for $(\theta,u)$ is a quasi-best approximation of the solution $u$ w.r.t.~$\|\bigcdot\|_X$.
These Galerkin approximations can be numerically computed by an inexact Uzawa type algorithm.

The uniform inf-sup condition has been demonstrated for finite element spaces $Y^\delta$ and $X^\delta$ w.r.t.~to partitions of the time-space cylinder that allow a subdivision into time-slabs. To circumvent this restriction, an a posteriori condition for quasi-optimality of $u^\delta$ has been derived on the size of $Y^\delta$ in relation to that of $X^\delta$.

The computational evaluation of this a posteriori condition requires an a posteriori estimator for the error in the Galerkin approximation from $Y^\delta$ of the solution $\theta_\delta \in Y$ of $A \theta_\delta=\ell-d_t u^\delta$.
The derivation of such an estimator for $A$ resulting from a quasi-linear parabolic PDE, and the presentation of numerical results for the resulting double-adaptivity algorithm will be postponed to forthcoming work.

\appendix
%%%%%%%%%%%%%%
\section{Optimal preconditioners in the tensor product setting} \label{sec:precond}
%%%%%%%%%%%%%%
For a domain $\Omega \subset \R^d$, let $H=L_2(\Omega)$, $V=H^1_0(\Omega)$, equipped with $\|\nabla \bigcdot\|_{L_2(\Omega)^d}$.
We set the Poincar\'{e}-Friedrich constant by $C_{\text{PF}}:=\sup_{0 \neq v \in H^1_0(\Omega)} \frac{\|v\|_{L_2(\Omega)}}{\|\nabla v\|_{L_2(\Omega)^d}}$.
For $\delta \in \Delta \setminus \delta_*$, let $X^\delta=X_t^\delta \otimes X_x^\delta$, $Y^\delta=Y_t^\delta \otimes X_x^\delta$ with $\gamma_\Delta\geq \gamma_{\Delta,t}\gamma_{\Delta,x}>0$ (cf.~\eqref{eq:tensor}).

To construct an optimal preconditioner for the Riesz map $X^\delta \rightarrow (X^\delta)'$, on $X^\delta$ we define the alternative (squared) norm
$$
\nrm \bigcdot\nrm_{X^\delta}^2:=\|E_X^\delta\bigcdot\|_Y^2+\|E_X^\delta\bigcdot\|_{H^1(\Xi) \otimes (X_x^\delta)'}^2,
$$
where $\|\bigcdot\|_{(X_x^\delta)'}:=\sup_{0 \neq v_x \in X_x^\delta} \frac{\langle\bigcdot,v_x\rangle_{L_2(\Omega)}}{\|\nabla v_x\|_{L_2(\Omega)^d}}$.
The scalar products corresponding to $\nrm \bigcdot\nrm_{X^\delta}$ and $\|\bigcdot\|_{(X_x^\delta)'}$ will be denoted by
 $\langle\!\langle \bigcdot,\bigcdot\rangle\!\rangle_{X^\delta}$ and $\langle\bigcdot,\bigcdot\rangle_{(X_x^\delta)'}$.

From
\begin{align*}
\|\bigcdot\|^2_{H^1(\Xi)\otimes (X_x^\delta)'} \leq \|\bigcdot\|^2_{H^1(\Xi)\otimes H^{-1}(\Omega)} &=
\|\bigcdot\|^2_{L_2(\Xi)\otimes H^{-1}(\Omega)}+\|\diff_t \bigcdot\|^2_{Y'}\\
& \leq C_{\text{PF}}^4 \|\bigcdot\|_Y^2+\|\diff_t \bigcdot\|^2_{Y'},
\end{align*}
we have
$$
\tfrac{\nrm\bigcdot\nrm_{X^\delta}^2}{1+C_{\text{PF}}^4}
\leq
\|E_X^\delta \bigcdot\|_Y^2+\|\diff_t E_X^\delta \bigcdot\|^2_{Y'}+\|\gamma_T E_X^\delta \bigcdot\|^2_{L_2(\Omega)}
=\|E_X^\delta \bigcdot\|^2_X \leq \tfrac{\|\bigcdot\|_{X^\delta}^2}{\gamma_\Delta^2}.
$$
On the other hand, with $C_\Xi$ defined in \eqref{eq:embedding},
\begin{align*}
\nrm\bigcdot\nrm_{X^\delta}^2 &\geq \|E_X^\delta\bigcdot\|_Y^2+\|\diff_t E_X^\delta\bigcdot\|_{L_2(\Xi) \otimes (X_x^\delta)'}^2 \geq \gamma_{\Delta,x}^2( \|E_X^\delta\bigcdot\|_Y^2+\|\diff_t E_X^\delta\bigcdot\|_{Y'}^2)\\
& \geq \tfrac{\gamma_{\Delta,x}^2}{1+C_\Xi^2} \|E_X^\delta\bigcdot\|_X^2 \geq \tfrac{\gamma_{\Delta,x}^2}{1+C_\Xi^2} \|\bigcdot\|_{X^\delta}^2,
\end{align*}
so that $\|\bigcdot\|_{X^\delta}$ and $\nrm\bigcdot\nrm_{X^\delta}$ are equivalent norms on $X^\delta$, uniformly in $\delta \in \Delta$.
\medskip

Let $R_{X^\delta}\colon X^\delta \rightarrow (X^\delta)'$ be the Riesz operator for $X^\delta$ when equipped with $\langle\!\langle \bigcdot,\bigcdot\rangle\!\rangle_{X^\delta}$.
Let $X^\delta$ be equipped with basis $\Phi^\delta$, formally viewed as a (column) vector, and $(X^\delta)'$ with the corresponding dual basis.
With $\cF_\delta\colon (X^\delta)' \rightarrow \R^{\dim X^\delta} \colon f \mapsto f(\Phi^\delta)$, and its adjoint $\cF_\delta'\colon \R^{\dim X^\delta} \rightarrow X^\delta\colon \vec{w} \mapsto \vec{w} \cdot \Phi^\delta$,
the \emph{matrix} ${\bf R}_X^\delta:=\cF_\delta R_{X^\delta} \cF_\delta'$ that represents $R_{X^\delta}$ is symmetric and positive definite. For ${\bf w}, {\bf v} \in \R^{\dim X^\delta}$, 
with $w:=\cF_\delta' {\bf w}$ and $v:=\cF_\delta' {\bf v}$, it holds that 
${\bf R}_X^\delta {\bf w} \cdot {\bf v} = (R_{X^\delta} w)(v)=\langle\!\langle w , v\rangle\!\rangle_{X^\delta}$.
For a symmetric and positive definite matrix
 ${\bf \tilde{R}}_X^\delta$, set $\tilde{R}_{X^\delta}:=\cF_\delta^{-1} {\bf \tilde{R}}_X^\delta (\cF_\delta')^{-1}$. Then
 $\frac{{\bf R}_X^\delta {\bf w} \cdot {\bf w}}{{\bf \tilde{R}}_X^\delta {\bf w} \cdot {\bf w}}=\frac{(R_{X^\delta} w)(w)}{(\tilde{R}_{X^\delta} w)(w)}$.
 So finding an optimal preconditioner for $R_{X^\delta}$ is equivalent to finding a matrix that is (uniformly) spectrally equivalent to ${\bf R}_X^\delta$, and whose inverse can be applied in linear complexity.
 
 Choosing $\Phi^\delta$ of tensor form $\Phi_t^\delta \otimes \Phi_x^\delta$, the matrix ${\bf R}_X^\delta$ is of the form
 $$
{\bf R}_X^\delta=  {\bf M}_t^\delta \otimes {\bf A}_x^\delta+({\bf M}_t^\delta + {\bf A}^\delta_t)\otimes {\bf M}_x^\delta ({\bf A}_x^\delta)^{-1}{\bf M}_x^\delta,
$$
where 
 ${\bf M}_t^\delta$ is the temporal mass matrix, 
 ${\bf A}_t^\delta$ is the temporal stiffness matrices,
 ${\bf M}_x^\delta$ is the spatial mass matrix, and 
  ${\bf A}_x^\delta$ is the spatial stiffness matrix, and where we used that
  $\|\bigcdot\|_{(X_x^\delta)'}=\|({\bf  A}_x^\delta)^{-\frac12} {\bf M}_x (\cF_\delta')^{-1}\bigcdot\|$.
  
In view of an efficient evaluation of operators, because of their local supports we have in mind $\Phi_t^\delta$ and $\Phi_x^\delta$ to be common nodal bases of finite element spaces.
Let $\Psi_t^\delta$ be an \emph{alternative} basis for $\Span \Phi_t^\delta$ that is (uniformly) stable  w.r.t.~\emph{both} $L_2(\Xi)$- and $H^1(\Xi)$-norm.
With $\widehat{\bf M}_t^\delta$ and  $\widehat{\bf A}_t^\delta$ denoting its corresponding mass and stiffness matrices, this means that 
 both $\widehat{\bf M}_t^\delta$ and $\widehat{\bf M}_t^\delta +\widehat{\bf A}^\delta_t$ are uniformly spectrally equivalent to their diagonals ${\bf D}^\delta_{t,1} \leq {\bf D}^\delta_{t,2}$. W.l.o.g.~we assume that ${\bf D}^\delta_{t,1} =\identity$.
 With $\widehat{\bf R}_X^\delta :=  \widehat{\bf M}_t^\delta \otimes {\bf A}_x^\delta+(\widehat{\bf M}_t^\delta + \widehat{\bf A}^\delta_t)\otimes {\bf M}_x^\delta ({\bf A}_x^\delta)^{-1}{\bf M}_x^\delta$, we conclude that with $\alpha_\psi:=\|\psi\|_{H^1(\Xi)}$, $\widehat{\bf R}_X^\delta$ is 
(uniformly) spectrally equivalent to
 $$
 \text{blockdiag}[{\bf A}_x^\delta+\alpha_\psi^2 {\bf M}_x^\delta ({\bf A}_x^\delta)^{-1}{\bf M}_x^\delta]_{\psi \in \Psi_t^\delta}.
 $$

 Since both ${\bf M}_x^\delta$ and ${\bf A}_x^\delta$ are symmetric positive definite, \cite[Thm.~4]{242.817} shows that
 $$
 \tfrac12 \big({\bf A}_x^\delta\!+\!\alpha_\phi^2 {\bf M}_x^\delta ({\bf A}_x^\delta)^{-1}{\bf M}_x^\delta\big)
 \!\leq\! 
 ({\bf A}_x^\delta\!+\!\alpha_\phi {\bf M}_x^\delta) ({\bf A}_x^\delta)^{-1} ({\bf A}_x^\delta\!+\!\alpha_\phi {\bf M}_x^\delta) 
 \!\leq\!
 {\bf A}_x^\delta\!+\!\alpha_\phi^2 {\bf M}_x^\delta ({\bf A}_x^\delta)^{-1}{\bf M}_x^\delta.
 $$
Consequently, if
 \be \label{eq:20}
 \widecheck{\bf R}_{\alpha_\psi}^\delta \eqsim {\bf A}_x^\delta+\alpha_\psi {\bf M}_x^\delta,
 \ee
then $(\widehat{\bf R}_X^\delta)^{-1}$ is 
(uniformly) spectrally equivalent to
 \be \label{eq:39}
  \text{blockdiag}[ (\widecheck{\bf R}_{\alpha_\psi}^\delta)^{-1} {\bf A}_x^\delta  (\widecheck{\bf R}_{\alpha_\psi}^\delta)^{-1}]_{\psi \in \Psi_t^\delta}.
 \ee
  
 Noticing that ${\bf A}_x^\delta+\alpha_\phi {\bf M}_x^\delta$ is the system matrix resulting from a diffusion-reaction problem, it is known that for $X_x^\delta$ being a finite element space, preconditioners $(\widecheck{\bf R}_{\alpha_\psi}^\delta)^{-1}$ for which \eqref{eq:20} is valid uniformly in $\alpha_\psi >0$ (and $\delta$), and whose application can be performed in linear complexity are provided by usual multigrid preconditioners \cite{241.1}. This holds true for finite element spaces w.r.t.~quasi-uniform meshes on domains $\Omega$ 
 where on $\{u \in H^1_0(\Omega)\colon \triangle u \in L_2(\Omega)\}$ it holds that
 $\|\Delta \bigcdot\|_{L_2(\Omega)} \eqsim \|\bigcdot\|_{H^2(\Omega)}$. It can be expected though that both these conditions are not necessary.
 
 Considering the construction of $\Psi_t^\delta$, wavelet collections are available that are (uniformly) stable  w.r.t.~both $L_2(\Xi)$- and $H^1(\Xi)$-norm, and that span 
 finite element spaces w.r.t.~partitions constructed by repeated uniform dyadic refinements (e.g.~see \cite[Fig.~2]{249.991} for the space of continuous piecewise linears). 
 With ${\bf T}^\delta$ denoting the wavelet-to-nodal basis transform, which can be applied in linear complexity, it holds that
 $\widehat{\bf R}_X^\delta=(({\bf T}^\delta)^\top \otimes \identity) {\bf R}_X^\delta ({\bf T}^\delta \otimes \identity)$. Knowing that $(\widehat{\bf R}_X^\delta)^{-1}$ is 
(uniformly) spectrally equivalent to the matrix in \eqref{eq:39}, we conclude that
 $$
( ({\bf T}^\delta) \otimes \identity)\,  \text{blockdiag}[ (\widecheck{\bf R}_{\alpha_\psi}^\delta)^{-1} {\bf A}_x^\delta  (\widecheck{\bf R}_{\alpha_\psi}^\delta)^{-1}]_{\psi \in \Psi_t^\delta}\,
( ({\bf T}^\delta)^\top \otimes \identity)
 $$ is the matrix representation of an optimal preconditioner for $R_X^\delta$.
 \medskip

Recalling that $Y^\delta=Y_t^\delta \otimes X_x^\delta$, the representation of the Riesz operator $R_{Y^\delta}\colon Y^\delta \rightarrow (Y^\delta)'$ is of the form
${\bf \widetilde{M}}_t^\delta \otimes {\bf A}_x^\delta$, where ${\bf A}_x^\delta$ is as above, and ${\bf \widetilde{M}}_t^\delta$ is the temporal mass matrix w.r.t.~the basis for $Y_t^\delta$. 
An optimal preconditioner for ${\bf A}_x^\delta$ is provided by multi-grid, whereas for a usual single scale basis for $Y^\delta_t$, $({\bf \widetilde{M}}_t^\delta)^{-1}$ can be applied exactly at linear cost.

%\bibliographystyle{halpha}
%\bibliography{../ref}
\end{document}